\documentclass[11pt, a4paper]{amsart}
\usepackage{cgeometry}
\usepackage{multirow}
\usepackage{manfnt}

\makeatletter
\def\hang{\hangindent\parindent}
\def\d@nger{\medbreak\begingroup\clubpenalty=10000
  \def\par{\endgraf\endgroup\medbreak} \noindent
  \hang\hangafter=-2
  \hbox to0pt{\hskip-\hangindent\dbend\hfill}}
\outer\def\danger{\d@nger}
\makeatother

\bibliography{ComplexGeometry}
\title{Transcendental b-divisors I --- Correspondence with currents}
\author{Mingchen Xia}

\begin{document}

\begin{abstract}
    We study transcendental b-divisors over compact Kähler manifolds. We establish the correspondence between closed positive currents and nef b-divisors.
    As an application, we establish the intersection theory of nef b-divisors, answering a question of Dang--Favre.
\end{abstract}

\maketitle

\tableofcontents

\section{Introduction}
In this paper, we study the transcendental b-divisors. In particular, we extend the intersection theory of algebraic b-divisors developed in \cite{DF20} to the transcendental setting. 

Let $X$ be a connected compact Kähler manifold of dimension $n$. Recall that an algebraic b-divisor (class) is an assignment $(\pi\colon Y\rightarrow X)\mapsto \mathbb{D}_Y\in \mathrm{NS}^1(Y)_{\mathbb{R}}$ (the real vector space spanned by the Néron--Severi group of $Y$), where $\pi$ runs over all modifications of $X$. These data are assumed to be compatible under push-forwards. An example is a so-called Cartier b-divisor, where we start with a modification $\pi\colon Y\rightarrow X$ and a class $\alpha$ on $Y$, the value of $\mathbb{D}_Z$ on any modification $Z\rightarrow X$ dominating $\pi$ is the pull-back of $\alpha$ to $Z$. The Cartier b-divisor is called nef if $\alpha$ can be taken as nef. In general, an algebraic b-divisor is nef if it can be approximated by nef Cartier b-divisors.

B-divisors generalize divisors while incorporating bimeromorphic twists. It is of interest to understand their intersection theory. When $X$ is projective, Dang–Favre \cite{DF20} established an intersection theory for nef b-divisors, which has been applied in dynamical systems \cite{DF20a} and K-stability \cite{XiaPPT}. Roughly speaking, they proved that in this case, a nef b-divisor can always be approximated by a \emph{decreasing} sequence of nef Cartier b-divisors. This result reduces the general intersection theory to that of Cartier b-divisors, which is essentially the same as the classical intersection theory as in \cite{Ful}. 

In the same paper, Dang--Favre asked the question of whether one can develop a similar theory for transcendental b-divisors, namely, when $X$ is not necessarily projective and when the $\mathbb{D}_Y$'s are just classes in $\mathrm{H}^{1,1}(Y,\mathbb{R})$. We give an affirmative answer in this paper. Our theory is new even when $X$ is projective.

The idea of the proof is already contained in the author's previous papers \cite{XiaPPT, Xia24}. Let us content ourselves to the algebraic setting for the moment. In this case, the two papers give an analytic approach to the Dang--Favre intersection theory. 
Suppose that $L_1,\ldots,L_n$ are big line bundles on $X$. Then for any singular Hermitian metrics $h_1,\ldots,h_n$ on $L_1,\ldots,L_n$, we can construct natural algebraic nef b-divisors $\mathbb{D}_1,\ldots,\mathbb{D}_n$ on $X$ using Siu's decomposition. A key result in these papers show that the Dang--Favre intersection number of $\mathbb{D}_1,\ldots,\mathbb{D}_n$ coincides with the mixed volume of $h_1,\ldots,h_n$.


Conversely, up to some technical details, any nef b-divisor essentially arises from Siu's decomposition of currents. In the algebraic setting, this was already known since due to \cite[Theorem~6.40]{BJ22} and \cite[Theorem~1.1]{DXZ23}, nef b-divisors and the $\mathcal{I}$-equivalence classes of (non-divisorial) currents are both in bijection with the same class of homogeneous non-Archimedean metrics. 
We prove the transcendental version:
\begin{theorem}
    There is a natural map $\mathbb{D}$ (via Siu's decomposition) sending each closed positive $(1,1)$-current $T$ on $X$ to a nef b-divisor $\mathbb{D}(T)$ over $X$.

    Let $\alpha$ be a modified nef cohomology class on $X$ and $\mathbb{D}$ be a nef and big b-divisor over $X$ with $\mathbb{D}_X=\alpha$, then there is a non-divisorial (namely, the generic Lelong number along any prime divisor on $X$ vanishes) closed positive $(1,1)$-current $T$ in $\alpha$ with $\mathbb{D}(T)=\mathbb{D}$. 
\end{theorem}
See \cref{def:volb} and \cref{def:nefb} for the precise definitions of nef and big b-divisors. The notion of modified nef cohomology classes is recalled  in \cref{subsec:modif}. The precise definition of $\mathbb{D}$ is recalled in \cref{subsec:b-div_curr}.

It is not clear to the author if the same holds when $\mathbb{D}$ is just nef but not big. 

The current $T$ is far from being unique. The degree of non-uniqueness is given by the next result. First recall that two closed positive  $(1,1)$-currents $T$ and $T'$ on $X$ are said to be $\mathcal{I}$-equivalent if they have the same Lelong numbers everywhere and the same remains true after pulling-back to any bimeromorphic modification of $X$. See \cref{def:Iequiv}.
\begin{theorem}\label{thm:corre_b_div_curr}
Let $\alpha$ be a modified nef cohomology class on $X$.
Two non-divisorial closed positive $(1,1)$-currents with positive volumes  in $\alpha$ correspond to the same nef b-divisor via $\mathbb{D}$ if and only if they are $\mathcal{I}$-equivalent.

In other words, there is a natural bijection $\mathbb{D}$ between the following sets:
    \begin{enumerate}
        \item The set of $\mathcal{I}$-equivalence classes of non-divisorial closed positive $(1,1)$-currents in $\alpha$ with positive volumes;
        \item the set of nef and big b-divisors $\mathbb{D}$ over $X$ with $\mathbb{D}_X=\alpha$.
    \end{enumerate}
\end{theorem}

The above two theorems correspond to \cref{thm:bdiv_current_corr}. Here modified nefness of $\alpha$ is not a strong constraint, since a general pseudoeffective class can always be decomposed into a modified nef part and a fixed part via the divisorial Zariski decomposition, see \cite{Bou02}.

Although \cref{thm:corre_b_div_curr} bears some resemblance with \cite[Proposition~3.1]{Tru24}, the contents of these theorems are radically different. In fact, Trusiani worked with b-divisors with possibly infinitely many components without passing to the numerical classes. 
It is a key observation in this paper that the numerical classes suffice to fully determine the singularities of a current modulo $\mathcal{I}$-equivalence, explaining the neatness of our theorem. At the same time, these theorems show that the author's notion of augmented b-divisors in \cite{Xia24} is unnecessary.

As a simple consequence of \cref{thm:corre_b_div_curr} and Demailly's approximation theorem, we deduce an analogue of \cite[Theorem~A]{DF20}.
\begin{corollary}[{\cref{cor:nef_Cartierseq}}]\label{cor:Demapp}
    Let $\mathbb{D}$ be a nef b-divisor over $X$. Then there is a  \emph{decreasing sequence} of nef and big Cartier b-divisors over $X$ with limit $\mathbb{D}$.
\end{corollary}
Based on \cref{cor:Demapp}, one could essentially repeat the arguments of \cite{DF20} to establish an intersection theory of nef b-divisors --- Reducing the general intersection theory of nef b-divisors to that of Cartier b-divisors, which is essentially known.

But we choose to follow a slightly different approach in \cref{sec:int_b_div}: We define the intersection number of nef b-divisors as the mixed analytic volume of the corresponding currents. Since the latter theory is already well-developed nowadays, this seems to be the most efficient way to establish the properties of the intersection product.

As a consequence, \cref{cor:Demapp} does not play a significant role in this paper, as Demailly's approximation is already a built-in feature of the analytic intersection theory of currents.

There is a technical subtlety here: There are at least two candidates for the analytic theory --- The $\mathcal{I}$-volumes developed in \cite{DX21, DX22} and the mixed volumes in the sense of Cao \cite{Cao14}. We will show in \cref{sec:mv} that these theories agree and satisfy the desired properties.

The idea of extending Dang--Favre’s theory using analytic methods is known to the author when he wrote \cite{XiaPPT} in 2020. What hinders the appearance of this paper is the slow development of the analytic theory\footnote{and France’s notoriously sluggish government, which made the author homeless for the most part of the year 2024}. The necessary tools were developed over years in \cite{DDNLmetric, DX22, DX21, Xia24, Xia21} and systematically summarized and extended in the author's book \cite{Xiabook}.

\cref{thm:corre_b_div_curr} sheds light on the algebraic theory as well. The powerful analytic machinery can be translated back to the algebraic theory, giving new insights even in the purely algebraic theory.
As an example, in \cref{sec:smpull} and \cref{sec:resbdiv},
we will study the two functorial operations of nef b-divisors: Along a smooth morphism, nef b-divisors can be pulled back; given a subvariety, a nef b-divisor can be restricted under a mild assumption. 
Both operations are essentially known in the algebraic setting via Boucksom--Jonsson's theory of non-Archimedean metrics. Our approach gives more straightforward definitions in terms of b-divisors themselves. This is not just for aesthetic reasons, it seems to the author that this is an essential step for attacking Collins--Tosatti's conjecture \cite{CT22}. Furthermore, the trace operator plays a key role in the monotonicity theorem \cite{Xiabdiv2}. 

Finally, all results in this paper hold for manifolds in Fujiki's class $\mathcal{C}$ as well. But for simplicity, we always restrict our discussion to Kähler manifolds.

\subsection*{Further directions}

For further developments of the intersection theory and applications to pluripotential theory, we refer to the sequel paper \cite{Xiabdiv2}.

Our intersection theory has direct applications in dynamic systems, as the algebraic theory of Dang--Favre did in \cite{DF20a}. 

In \cite{DF20}, Dang--Favre suggested that defining the intersection theory may rely on the transcendental Morse inequality conjectured in \cite{BDPP13}. We expect that conversely, our theory should be helpful for understanding this conjecture. Less ambitiously, our theory should be helpful when trying to understand a related result regarding the non-Kähler locus, as conjectured in \cite{CT22}. 

Finally, there is a notion of Okounkov bodies of transcendental nef b-divisors, exactly as the algebraic case studied in \cite[Section~11.3]{Xiabook}. Over $\mathbb{C}$, the theory is well-understood using the analytic theory.
However, it is not clear how to work out similar results based on the theory of Dang--Favre or Boucksom--Jonsson over general base fields.

\subsection*{Editorial note}
An earlier version of this paper has been available on the author's webpage since early 2025. 
The submitted version relies on the first version of the author's monograph \cite{Xiabook}. 
The monograph underwent a thorough revision in September 2025 and part of the results proved in this paper were reproduced in the book.
The current version of the paper is almost identical to the submitted version, apart from minor modifications reflecting the revision of the book.

\subsection*{Acknowledgments} The author would like to thank Nicholas McCleerey, Charles Favre and Antonio Trusiani for their comments on the draft. 
Part of the work was carried out during the author's visit to Yunnan Normal University in 2024, the author would like to thank Prof.~Zhipeng Yang for his hospitality and the support of Yunnan Key Laboratory of Modern Analytical Mathematics and Applications (No. 202302AN360007).
The author was initially supported by the Knut och Alice Wallenbergs Stiftelse KAW 2024.0273, then by the National Key R\&D Program of China 2025YFA1018200.

\section{Preliminaries}

Let $X$ be a connected compact Kähler manifold of dimension $n$. 

\subsection{Modifications and cones}\label{subsec:modif}
In this paper, we use the word \emph{modification} in a very non-standard sense.
\begin{definition}
    A \emph{modification} of $X$ is a bimeromorphic morphism $\pi\colon Y\rightarrow X$, which is a finite composition of blow-ups with smooth centers.

        We say a modification $\pi'\colon Z\rightarrow X$ \emph{dominates} another $\pi\colon Y\rightarrow X$ if there is a morphism $g\colon Z\rightarrow Y$ making the following diagram commutative:
    \begin{equation}\label{eq:domi}
    \begin{tikzcd}
Z \arrow[rr,"g"] \arrow[rd, "\pi'"'] &   & Y \arrow[ld, "\pi"] \\
                                 & X. &                    
\end{tikzcd}
    \end{equation}
\end{definition}
Note that $\pi$ is necessarily projective and $Y$ is always a Kähler manifold. 

The modifications of $X$ together with the domination relation form a directed set $\Modif(X)$. 

Given classes $\alpha,\beta\in \mathrm{H}^{1,1}(X,\mathbb{R})$, we say $\alpha\leq \beta$ if $\beta-\alpha$ is pseudoeffective.

\begin{proposition}\label{prop:blowup_cone}
    Let $\pi\colon Y\rightarrow X$ be a blow-up with \emph{connected} smooth center of codimension at least $2$ with exceptional divisor $E$. Then there is a natural identification 
    \begin{equation}\label{eq:H11blowup}
    \mathrm{H}^{1,1}(Y,\mathbb{R})=\mathrm{H}^{1,1}(X,\mathbb{R})\oplus \mathbb{R}\{E\}.
    \end{equation}
\end{proposition}
See \cite{RYY19} for a much more general result. In general, the pseudoeffective cone of $Y$ does not admit any simple descriptions.

Fix a reference Kähler form $\omega$ on $X$.
Recall that a class $\alpha\in \mathrm{H}^{1,1}(X,\mathbb{R})$ is \emph{modified nef} (or \emph{movable}) if for any $\epsilon>0$, we can find a closed $(1,1)$-current $T\in \alpha$ such that
\begin{enumerate}
    \item $T+\epsilon\omega\geq 0$;
    \item $\nu(T+\epsilon\omega,D)=0$ for any prime divisor $D$ on $X$.
\end{enumerate}
This definition is independent of the choice of $\omega$.
Here $\nu(\bullet,D)$ denote the generic Lelong number along $D$.

These classes are called \emph{nef en codimension $1$} in Boucksom's thesis \cite{Bou02}, where they were introduced for the first time. Modified nef classes form a closed convex cone in $\mathrm{H}^{1,1}(X,\mathbb{R})$. Note that a modified nef class is necessarily pseudoeffective. A nef class is obviously modified nef.

Recall the multiplicity of a cohomology class as defined in \cite[Section~2.1.3]{Bou02}.
\begin{definition}
    Let $\alpha\in \mathrm{H}^{1,1}(X,\mathbb{R})$ be a pseudoeffective class and $D$ be a prime divisor on $X$. We define the Lelong number $\nu(\alpha,D)$ as follows:
    \begin{enumerate}
        \item When $\alpha$ is big, define $\nu(\alpha,D)=\nu(T,D)$ for any closed positive $(1,1)$-current $T\in \alpha$ with minimal singularities (namely, a current in $\alpha$ that is less singular than any current in $\alpha$).
        \item In general, define
        \[
        \nu(\alpha,D)\coloneqq \lim_{\epsilon\to 0+}\nu(\alpha+\epsilon \{\omega\},D).
        \]
    \end{enumerate}
\end{definition}
When $\alpha$ is big, (2) is compatible with (1) and the definition is independent of the choice of $\omega$.
By definition, a pseudoeffective class $\alpha$ is modified nef if and only if $\nu(\alpha,D)=0$ for all prime divisors $D$ on $X$.

Let us recall the behavior of several cones under modifications.
\begin{proposition}\label{prop:cones_modif}
    Let $\pi\colon Y\rightarrow X$ be a proper bimeromorphic morphism from a Kähler manifold $Y$. 
    \begin{enumerate}
        \item For any nef class $\alpha\in \mathrm{H}^{1,1}(X,\mathbb{R})$, $\pi^*\alpha$ is nef.
        \item For any modified nef class $\beta\in \mathrm{H}^{1,1}(Y,\mathbb{R})$, $\pi_*\beta$ is modified nef.
        \item For any big class $\alpha\in \mathrm{H}^{1,1}(X,\mathbb{R})$, $\pi^*\alpha$ is big. Moreover, $\vol \pi^*\alpha=\vol \alpha$.
        \item For any big class $\beta\in \mathrm{H}^{1,1}(Y,\mathbb{R})$, $\pi_*\beta$ is big. Moreover, $\vol \pi_*\beta\geq \vol \beta$.
    \end{enumerate}
\end{proposition}
\begin{proof}
    Only (2) requires a proof. Fix a Kähler class $\gamma$. Replacing $\beta$ by $\beta+\epsilon\gamma$ for $\epsilon\in (0,1)$, we reduce immediately to the case where $\beta$ is big as well. Let $T$ (resp. $S$) be a current with minimal singularities in $\pi_*\beta$ (resp. in $\beta$) and $D$ be a prime divisor on $X$, it suffices to show that
    \[
    \nu(T,D)=0,
    \]
    by \cref{lma:ndvpush} below, $\nu(\pi_*S,D)=0$, so our assertion follows.
\end{proof}

Let $T$ be a closed positive $(1,1)$-current on $X$. Then we define the \emph{regular part} $\Reg T$ of $T$ as the regular part of $T$ with respect to Siu's decomposition. In other words, we write
\begin{equation}\label{eq:Siudec}
T=\Reg T+\sum_i c_i [E_i],
\end{equation}
where $E_i$ is a countable collection of prime divisors on $X$ and $c_i=\nu(T,E_i)>0$; the regular part $\Reg T$ is a closed positive $(1,1)$-current whose generic Lelong number along each prime divisor on $X$ is $0$.

\begin{definition}\label{def:nd}
    We say a closed positive $(1,1)$-current $T$ on $X$ is \emph{non-divisorial} (resp. \emph{divisorial}) if $T=\Reg T$ (resp. $\Reg T=0$).
\end{definition}
Note that the cohomology class of a non-divisorial current is always modified nef. Conversely, a current with minimal singularities in a \emph{big} and modified nef class is always non-divisorial. 

There is a closely related notion introduced in \cite{McC21}:
\begin{definition}
    We say a closed positive $(1,1)$-current $S$ on $X$ is \emph{non-pluripolar} (resp. \emph{pluri-supported}) if $S=\langle S \rangle$ (resp. $\langle S \rangle=0$).
\end{definition}
Here $\langle S \rangle$ denotes the non-pluripolar polar part of $S$, namely the non-pluripolar product of $S$ itself in the sense of \cite{BEGZ10}. 

Clearly, a divisorial current is necessarily pluri-supported, and a non-pluripolar current is necessarily non-divisorial. But these notions are not equivalent in general. However, within the class of $\mathcal{I}$-good singularities, these notions turn out to be equivalent. Since we are not in need of the latter result in this paper, we omit the proof.

\begin{lemma}\label{lma:ndvpush}
    Let $\pi\colon Y\rightarrow X$ be a proper bimeromorphic morphism from Kähler manifold $Y$. Let $T$ be a non-divisorial current on $Y$, then $\pi_*T$ is non-divisorial.
\end{lemma}
Conversely, if $S$ is a non-divisorial current on $X$, $\pi^*S$ is could have divisorial part. As a simple example, consider $S$ on $\mathbb{P}^2$, whose local potential near $0\in \mathbb{C}^2_{z,w}$ looks like $\log (|z|^2+|w|^2)$.
\begin{proof}
    Let $D$ be a prime divisor on $X$. It follows from Zariski's main theorem (\cite[Théorème~1.7]{Dem85}) that $D$ is not contained in the exceptional locus of $\pi$.
    Let $D'$ be the strict transform of $D$. Thanks to Siu's semicontinuity theorem, we have
    \[
    \nu(\pi_*T,D)=\nu(T,D')=0.
    \]
    Hence $\pi_*T$ is non-divisorial.
\end{proof}

\subsection{Quasi-plurisubharmonic functions}

We first recall the notions of $P$ and $\mathcal{I}$-equivalences. The latter is introduced in \cite{DX22} based on \cite{BFJ08}. The former was introduced in \cite{Xiabook} based on \cite{RWN14}.

\begin{definition}
    Let $\varphi,\psi$ be quasi-plurisubharmonic functions on $X$. We say $\varphi\sim_P \psi$ (resp. $\varphi\preceq_P \psi$) if there is a closed smooth real $(1,1)$-form $\theta$ on $X$ such that $\varphi,\psi\in \PSH(X,\theta)_{>0}$ and 
    \[
    P_{\theta}[\varphi]=P_{\theta}[\psi]\quad (\textup{resp.} P_{\theta}[\varphi]\leq P_{\theta}[\psi]).
    \]
\end{definition}
Here $\PSH(X,\theta)$ denotes the space of $\theta$-plurisubharmonic functions on $X$ and $\PSH(X,\theta)_{>0}$ denotes the subset consisting of $\varphi\in \PSH(X,\theta)$ with $\int_X \theta_{\varphi}^n>0$, with $\theta_{\varphi}=\theta+\ddc\varphi$. 
Here and in the sequel, the Monge--Ampère type product $\theta_{\varphi}^n$ is always understood in the non-pluripolar sense of \cite{BT87, GZ07, BEGZ10}. The envelope $P_{\theta}$ is defined as follows:
\[
P_{\theta}[\varphi]\coloneqq \sups_{\!\!\! C\in \mathbb{R}} (\varphi+C)\land 0,
\]
where $(\varphi+C)\land 0$ is the maximal element in $\PSH(X,\theta)$ dominated by both $\varphi+C$ and $0$.

Given a closed smooth real $(1,1)$-form $\theta$ on $X$ so that $\varphi,\psi\in \PSH(X,\theta)$, we also say $\theta_{\varphi}\sim_P \theta_{\psi}$ (resp. $\theta_{\varphi}\preceq_P \theta_{\psi}$) if $\varphi\sim_P \psi$ (resp. $\varphi\preceq_P \psi$). The same convention applies also to the $\mathcal{I}$-partial order introduced later.

The main interest of the $P$-partial order lies in the following monotonicity theorem.
\begin{theorem}\label{thm:mono2}
    Let $\theta_1,\ldots,\theta_n$ be closed real smooth $(1,1)$-forms on $X$.
    Let $\varphi_i,\psi_i\in \PSH(X,\theta_i)$ for $i=1,\ldots,n$. Assume that $\varphi_i\preceq_P \psi_i$ for each $i=1,\ldots,n$. Then
    \[
        \int_X \theta_{1,\varphi_1}\wedge \cdots \wedge  \theta_{n,\varphi_n}\leq \int_X \theta_{1,\psi_1}\wedge \cdots \wedge  \theta_{n,\psi_n}.
    \]
\end{theorem}
This result is a consequence of the monotonicity theorem of Witt Nyström \cite{WN19, DDNL18mono}. See \cite[Proposition~6.1.4]{Xiabook} for the proof.

\begin{definition}\label{def:Iequiv}
    Let $\varphi,\psi$ be quasi-plurisubharmonic functions on $X$. We say $\varphi\sim_{\mathcal{I}} \psi$ (resp. $\varphi\preceq_{\mathcal{I}}\psi$) if $\mathcal{I}(\lambda\varphi)=\mathcal{I}(\lambda\psi)$ (resp. $\mathcal{I}(\lambda\varphi)\subseteq\mathcal{I}(\lambda\psi)$) for all real $\lambda>0$. 
\end{definition}
Here $\mathcal{I}$ denotes the multiplier ideal sheaf in the sense of Nadel. 

If $\theta$ is a closed smooth real $(1,1)$-form such that $\varphi,\psi\in \PSH(X,\theta)$, then $\varphi\preceq_{\mathcal{I}} \psi$ if and only if
\[
P_{\theta}[\varphi]_{\mathcal{I}}\leq P_{\theta}[\psi]_{\mathcal{I}},
\]
where
\[
P_{\theta}[\varphi]_{\mathcal{I}}=\sup\left\{\eta\in \PSH(X,\theta):\eta\leq 0, \mathcal{I}(\lambda\varphi)\supseteq \mathcal{I}(\lambda\eta)\textup{ for all }\lambda>0 \right\}.
\]
Equivalently, we may replace $\supseteq$ by $=$ in this equation.

Another equivalent formulation of \cref{def:Iequiv} is that for any prime divisor $E$ over $X$, we have
\[
\nu(\varphi,E)=\nu(\psi,E)\quad \textup{resp. } \nu(\varphi,E)\geq\nu(\psi,E).
\]
Here $\nu$ denotes the generic Lelong number. We refer to \cite[Section~3.2.1]{Xiabook} for the details.

Given any $\varphi\in \PSH(X,\theta)$, we have 
\[
\varphi-\sup_X \varphi\leq P_{\theta}[\varphi]\leq P_{\theta}[\varphi]_{\mathcal{I}}.
\]
See \cite[Proposition~2.18]{DX22} or \cite[Proposition~3.2.9]{Xiabook}.

For later use, let us recall the following:
\begin{lemma}\label{lma:Ipartorderinv}
Let $\pi\colon Y\rightarrow X$ be a proper bimeromorphic morphism from a Kähler manifold $Y$.
    Given two quasi-plurisubharmonic functions $\varphi,\psi$ on $X$, then the following are equivalent:
    \begin{itemize}
        \item $\varphi\preceq_{\mathcal{I}} \psi$;
        \item $\pi^*\varphi\preceq_{\mathcal{I}} \pi^*\psi$.
    \end{itemize}
\end{lemma}
\begin{proof}
    (1) $\implies$ (2). Just observe that each prime divisor over $Y$ is also a prime divisor over $X$.

    (2) $\implies$ (1). This follows from the well-known formula:
    \[
    \pi_*\left( \omega_{Y/X}\otimes \mathcal{I}(\lambda \pi^*\varphi) \right)=\mathcal{I}(\lambda\varphi),\quad \lambda>0,
    \]
    where $\omega_{Y/X}$ is the relative dualizing sheaf. See \cite[Proposition~5.8]{Dem12}.
\end{proof}

The operation $P_{\theta}[\bullet]_{\mathcal{I}}$ is idempotent. We say $\varphi\in \PSH(X,\theta)$ is \emph{$\mathcal{I}$-model} if $P_{\theta}[\varphi]_{\mathcal{I}}=\varphi$. Similarly, on the subset $\PSH(X,\theta)_{>0}$, the operation $P_{\theta}[\bullet]$ is also idempotent, see \cite[Theorem~3.12]{DDNL18mono}. We say $\varphi\in \PSH(X,\theta)_{>0}$ is \emph{model} if $P_{\theta}[\varphi]=\varphi$.

Suppose that $\{\theta\}$ is big.
It is shown in \cite{DDNLmetric} that there is a pseudometric $d_S$ on $\PSH(X,\theta)$ satisfying the following inequality:
For any $\varphi,\psi\in \PSH(X,\theta)$, we have
    \begin{equation}\label{eq:ds_biineq}
        \begin{split}
            d_S(\varphi,\psi)\leq & \frac{1}{n+1}\sum_{j=0}^n \left( 2\int_X \theta_{\varphi\lor \psi}^j\wedge \theta_{V_{\theta}}^{n-j}-\int_X \theta_{\varphi}^j\wedge \theta_{V_{\theta}}^{n-j}-\int_X \theta_{\psi}^j\wedge \theta_{V_{\theta}}^{n-j} \right) \\
            \leq & C_n d_S(\varphi,\psi),
        \end{split}
\end{equation}
where $C_n=3(n+1)2^{n+2}$. Here $V_{\theta}=\max\{\varphi\in \PSH(X,\theta):\varphi\leq 0\}$.
Moreover, $d_S(\varphi,\psi)=0$ if and only if $\varphi\sim_P \psi$. See \cite[Proposition~6.2.2]{Xiabook}. In particular, the $d_S$-pseudometric descends to a pseudometric (still denoted by $d_S$) on the space of closed positive $(1,1)$-currents in $\{\theta\}$. 

Given a net of closed positive $(1,1)$-currents $T_i$ in $\{\theta\}$, and another closed positive $(1,1)$-current $T$ in $\{\theta\}$. It is shown in \cite[Section~4]{Xia21} and \cite[Corollary~6.2.8]{Xiabook} that $T_i\xrightarrow{d_S}T$ if and only if $T_i+\omega\xrightarrow{d_S}T+\omega$ for any Kähler form $\omega$ on $X$.

In general, given closed positive $(1,1)$-currents $T_i$ and $T$ on $X$, we say $T_i\xrightarrow{d_S} T$ if we can find Kähler forms $\omega_i$ and $\omega$ on $X$ such that the $T_i+\omega_i$'s and $T+\omega$ represent the same cohomology class and $T_i+\omega_i\xrightarrow{d_S} T+\omega$. This definition is independent of the choices of the $\omega_i$'s and $\omega$.

We introduce a stronger notion in this paper:
\begin{definition}\label{def:stronconve}
    Let $(T_i)_i$ be a net of closed positive $(1,1)$-current on $X$ and $T$ be a closed positive $(1,1)$-current on $X$. We say $T_i\implies T$ if 
    \begin{enumerate}
        \item $T_i\xrightarrow{d_S} T$;
        \item $\{T_i\}\to \{T\}$.
    \end{enumerate}
\end{definition}

A quasi-plurisubharmonic function $\varphi$ on $X$ is called \emph{$\mathcal{I}$-good} if there is a closed smooth real $(1,1)$-form $\theta$ on $X$ such that $\varphi\in \PSH(X,\theta)_{>0}$ and 
\[
P_{\theta}[\varphi]=P_{\theta}[\varphi]_{\mathcal{I}}.
\]
For any closed smooth real $(1,1)$-form $\theta'$ on $X$ so that $\theta'+\ddc\varphi\geq 0$, we also say the current $\theta'_{\varphi}$ is $\mathcal{I}$-good. This notion is independent of the choice of $\theta$, as proved in \cite[Lemma~1.7]{Xia24}. See also \cite[Section~7.1]{Xiabook}. As a simple example, an $\mathcal{I}$-model potential with positive non-pluripolar mass is always $\mathcal{I}$-good.

A key result proved in \cite{DX22,DX21} is the following:
\begin{theorem}\label{thm:DX}
    A closed positive $(1,1)$-current $T$ on $X$ is $\mathcal{I}$-good if and only if there is a sequence of closed positive $(1,1)$-currents $(T_j)_j$ on $X$ with analytic singularities such that $T_j\implies T$.

    In fact, $(T_j)_j$ can be taken as any quasi-equisingular approximation of $T$.
\end{theorem}
Here we say a closed positive $(1,1)$-current $T$ has analytic singularities if locally $T$ can be written as $\ddc f$, where $f$ is a plurisubharmonic function of the following form:  
\[
c\log (|f_1|^2+\cdots +|f_N|^2)+R, 
\]
where $c\in \mathbb{Q}_{\geq 0}$, $f_1,\ldots,f_N$ are holomorphic functions on $X$ and $R$ is a bounded function. A few subtleties of this notion are discussed in \cite[Remark~2.7]{DRWNXZ}.  When we write $T=\theta+\ddc\varphi$ for some smooth closed real $(1,1)$-form $\theta$ and $\varphi\in \PSH(X,\theta)$, we also say $\varphi$ has analytic singularities. 

As a particular case, if $D$ is an effective $\mathbb{Q}$-divisor on $X$, we say a closed positive $(1,1)$-current $T$ has \emph{log singularities} along $D$ if $T-[D]$ is positive, and has locally bounded potentials. It is easy to see that $T$ has analytic singularities. Conversely, if we begin with $T$ with analytic singularities, there is always a modification $\pi\colon Y\rightarrow X$ so that $\pi^*T$ has log singularities along an effective $\mathbb{Q}$-divisor on $Y$. See \cite[Page~104]{MM07}.

Let $\theta$ be a smooth closed real $(1,1)$-form on $X$ and $\eta\in \PSH(X,\theta)$. We say a sequence $(\eta^j)_j$ of quasi-plurisubharmonic functions is a \emph{quasi-equisingular approximation} of $\eta$ if the following are satisfied:
\begin{enumerate}
    \item for each $j$, $\eta^j$ has analytic singularities;
    \item $(\eta^j)_j$ is decreasing with limit $\eta$;
    \item for each $\lambda'>\lambda>0$, we can find $j_0>0$ so that for $j\geq j_0$,
    \[
    \mathcal{I}(\lambda'\eta^j)\subseteq \mathcal{I}(\lambda \eta);
    \]
    \item There is a decreasing sequence $(\epsilon_j)_j$ in $\mathbb{R}_{\geq 0}$ with limit $0$, and a Kähler form $\omega$ on $X$ so that 
    \[
        \eta^j\in \PSH\left(X,\theta+\epsilon_j\omega\right)
    \]
    for each $j>0$.
\end{enumerate}
The existence of quasi-equisingular approximations is guaranteed by \cite{DPS01}. We also say $(\theta+\ddc\eta^j)_j$ is a quasi-equisingular approximation of $\theta+\ddc \eta$.

The class $\mathcal{I}$-good singularities is closed under many natural operations.
\begin{proposition}\label{prop:Igoodclosed}
    The sum and maximum of two $\mathcal{I}$-good quasi-plurisubharmonic functions are still $\mathcal{I}$-good. If $\theta$ is a closed real smooth $(1,1)$-form on $X$ and $(\varphi_i)_i$ is a non-empty bounded from above family of $\mathcal{I}$-good $\theta$-psh functions, then $\sups_i \varphi_i$ is also $\mathcal{I}$-good.
\end{proposition}
See \cite[Section~7.2]{Xiabook} for the proofs.

\section{Mixed volumes}\label{sec:mv}
Let $X$ be a connected compact Kähler manifold of dimension $n$. Let $T_1,\ldots,T_n$ be closed positive $(1,1)$-currents on $X$. Let $\theta_1,\ldots,\theta_n$ be closed real smooth $(1,1)$-forms on $X$ in the cohomology classes of $T_1,\ldots,T_n$ respectively. Consider $\varphi_i\in \PSH(X,\theta_i)$ so that $T_i=\theta_i+\ddc \varphi_i$ for each $i=1,\ldots,n$. Fix a reference K\"ahler form $\omega$ on $X$.

\subsection{The different definitions}

For each $i=1,\ldots,n$, let $(\varphi_i^j)_j$ be a quasi-equisingular approximation of $\varphi_i$.  
\begin{definition}
    The mixed volume of $T_1,\ldots,T_n$ in the sense of Cao is defined as follows:
    \[
    \langle T_1,\ldots,T_n\rangle_C\coloneqq \lim_{j\to\infty}\int_X \left(\theta_1+\epsilon_j\omega+\ddc \varphi_1^j\right)\wedge \cdots \wedge \left(\theta_n+\epsilon_j\omega+\ddc \varphi_n^j\right),
    \]
    where $(\epsilon_j)_j$ is a decreasing sequence with limit $0$ such that $\varphi_i^j\in \PSH(X,\theta_i+\epsilon_j \omega)$ for each $i=1,\ldots,n$ and $j>0$.
\end{definition}

It is shown in \cite{Cao14} Section~2 that this definition is independent of the choices of the $\theta_i$'s, the $\epsilon_j$'s, the $\varphi_i$'s, the $\varphi_i^j$'s and $\omega$.

A different definition relies on the $\mathcal{I}$-envelope technique studied in \cite{DX21, DX22}. 
Recall that the volume of a current is defined in \cite[Definition~3.2.3]{Xiabook}:
\[
\vol (\theta+\ddc\varphi)=\int_X \left(\theta+\ddc P_{\theta}[\varphi]_{\mathcal{I}}\right)^n.
\]
It depends only on the current $\theta+\ddc \varphi$, not on the choice of the choices of $\theta$ and $\varphi$. In general, as shown in \cite{DX22, DX21},
\[
\vol (\theta+\ddc\varphi)\geq \int_X (\theta+\ddc\varphi)^n.
\]
If furthermore the right-hand side is positive, then the equality holds if and only if $\varphi$ is $\mathcal{I}$-good. We refer to \cite[Section~7.1]{Xiabook} for the details.
\begin{definition}
Assume that $\vol T_i>0$ for all $i=1,\ldots,n$.
    The mixed volume of $T_1,\ldots,T_n$ in the sense of Darvas--Xia is defined as follows:
    \begin{equation}\label{eq:volmixed}
    \vol(T_1,\ldots,T_n)=\int_X \left(\theta_1+\ddc P_{\theta_1}[\varphi_1]_{\mathcal{I}}\right)\wedge \dots\wedge \left(\theta_n+\ddc P_{\theta_n}[\varphi_n]_{\mathcal{I}}\right).
    \end{equation}

In general, define
\begin{equation}\label{eq:volmixedgeneral}
\vol(T_1,\ldots,T_n)=\lim_{\epsilon\to 0+}\vol(T_1+\epsilon\omega,\ldots,T_n+\epsilon\omega).
\end{equation}

\end{definition}
This definition is again independent of the choices of $\omega$, the $\theta_i$'s and the $\varphi_i$'s, using the same proof as \cite[Proposition~3.2.7]{Xiabook}.

The mixed volume can be regarded as generalizations of the movable intersection theory. In fact, when each $T_i$ has minimal singularities, the mixed volume is exactly the movable intersection of corresponding cohomology classes.

When $\vol T_i>0$ for all $i=1,\ldots,n$, the definition \eqref{eq:volmixedgeneral} is compatible with \eqref{eq:volmixed}, as from the $\mathcal{I}$-goodness of the $P_{\theta_i}[\varphi_i]_{\mathcal{I}}$'s, we have
\[
P_{\theta_i+\epsilon\omega}\left[ P_{\theta_i}[\varphi_i]_{\mathcal{I}}\right]=P_{\theta_i+\epsilon\omega}[\varphi_i]_{\mathcal{I}}.
\]
Hence \eqref{eq:volmixedgeneral} reduces to \eqref{eq:volmixed} as a consequence of \cref{thm:mono2}.

When $T_1=\dots=T_n=T$, the above definition is compatible with pure case:
\begin{proposition}
    We always have
    \[
    \vol (T,\ldots,T)=\vol T.
    \]
\end{proposition}
\begin{proof}
    Write $T=\theta_{\varphi}$. In more concrete terms, we need to show that
    \[
    \lim_{\epsilon\to 0+}\int_X (\theta+\epsilon\omega+\ddc P_{\theta+\epsilon\omega}[\varphi]_{\mathcal{I}})^n=\int_X (\theta+\ddc P_{\theta}[\varphi]_{\mathcal{I}})^n.
    \]

     We may replace $\varphi$ by $P_{\theta}[\varphi]_{\mathcal{I}}$ and assume that $\varphi$ is $\mathcal{I}$-model in $\PSH(X,\theta)$. Then we claim that
    \[
    \varphi=\inf_{\epsilon>0}P_{\theta+\epsilon\omega}[\varphi]_{\mathcal{I}}.
    \]
    From this, our assertion follows from \cite[Proposition~3.1.9]{Xiabook}.
    
    The $\leq$ direction is clear. For the converse, it suffices to show that for each prime divisor $E$ over $X$, we have
    \[
    \nu(\varphi,E)\leq \nu\left(\inf_{\epsilon>0}P_{\theta+\epsilon\omega}[\varphi]_{\mathcal{I}},E\right).
    \]
    We simply compute
    \[
    \nu\left(\inf_{\epsilon>0}P_{\theta+\epsilon\omega}[\varphi]_{\mathcal{I}},E\right)\geq \sup_{\epsilon>0}\nu\left(P_{\theta+\epsilon\omega}[\varphi]_{\mathcal{I}},E\right)=\nu(\varphi,E).
    \]
\end{proof}

\begin{proposition}\label{prop:multilinear}
    Both volumes are symmetric. The mixed volume in the sense of Cao is multi-$\mathbb{Q}_{\geq 0}$-linear, while the mixed volume in the sense of Darvas--Xia is multi-$\mathbb{R}_{\geq 0}$-linear.
\end{proposition}
    The multi-$\mathbb{Q}_{\geq 0}$-linearity means two things: 
    \begin{enumerate}
        \item For each $\lambda\in \mathbb{Q}_{\geq 0}$, we have
        \[
        \langle \lambda T_1,T_2,\ldots,T_n\rangle_C=\lambda\langle T_1,T_2,\ldots,T_n\rangle_C.
        \]
        \item  If $T_1'$ is anther closed positive $(1,1)$-current, then
        \begin{equation}\label{eq:multilinear}
        \langle T_1+T_1',T_2,\ldots,T_n\rangle_C=\langle T_1,T_2,\ldots,T_n\rangle_C+\langle T_1',T_2,\ldots,T_n\rangle_C.
        \end{equation}
    \end{enumerate}
    Multi-$\mathbb{R}_{\geq 0}$-linearity is defined similarly.
\begin{proof}
    We first handle the mixed volumes in the sense of Cao.
    Only the property \eqref{eq:multilinear} needs a proof. But this follows from the fact that the sum of two quasi-equisingular approximations is again a quasi-equisingular approximation. See \cite[Theorem~6.2.2, Corollary~7.1.2]{Xiabook}.

    Next we handle the case of mixed volumes in the sense of Darvas--Xia.
    We only need to show that
    \begin{equation}\label{eq:voladd}
    \vol (T_1+T_1',T_2,\ldots,T_n)=\vol(T_1,T_2,\ldots,T_n)+\vol(T_1',T_2,\ldots,T_n).
    \end{equation}
    Thanks to the definition \eqref{eq:volmixedgeneral}, we may assume that $\vol T_i>0$ for each $i$ and $\vol T_1'>0$. Write $T_1'=\theta_1'+\ddc \varphi_1'$.
    Then thanks to \cref{prop:Igoodclosed},
    \[
    P_{\theta_1}[\varphi_1]_{\mathcal{I}}+P_{\theta_1'}[\varphi_1']_{\mathcal{I}}\sim_P P_{\theta_1+\theta_1'}[\varphi_1+\varphi_1']_{\mathcal{I}}.
    \]
    Therefore, \eqref{eq:voladd} follows from \cref{thm:mono2}.
\end{proof}

\begin{theorem}\label{thm:CaoequalDX}
    We have 
    \begin{equation}\label{eq:CaoeqDX}
        \langle T_1,\ldots,T_n\rangle_C=\vol(T_1,\ldots,T_n).
    \end{equation}
\end{theorem}
In particular, we no longer need the notation $ \langle T_1,\ldots,T_n\rangle_C$.
\begin{proof}
    \textbf{Step~1}. We reduce to the case where $T_1=\dots=T_n$. 

    Suppose this special case has been proved. Let $\lambda_1,\dots,\lambda_n\in \mathbb{Q}_{>0}$ be some numbers. Then
    \[
    \left\langle \sum_{i=1}^n \lambda_i T_i,\dots, \sum_{i=1}^n \lambda_i T_i \right\rangle_C=\vol \left(\sum_{i=1}^n \lambda_i T_i \right).
    \]
    It follows from \cref{prop:multilinear} that both sides are polynomials in the $\lambda_i$'s. Comparing the coefficients of $\lambda_1\cdots\lambda_n$, we conclude \eqref{eq:CaoeqDX}.

    From now on, we assume that $T_1=\dots=T_n=T$. Write $T=\theta_{\varphi}$.
    
    \textbf{Step~2}. We reduce to the case where $T$ is a K\"ahler current. For this purpose, it suffices to show that
    \[
    \lim_{\epsilon\to 0+}\langle T_1+\epsilon\omega,\ldots,T_n+\epsilon\omega\rangle_C=\langle T_1,\ldots,T_n\rangle_C,
    \]
    which is obvious by definition.

    \textbf{Step~3}. Let $(\varphi^j)_j$ be a quasi-equisingular approximation of $\varphi$ in $\PSH(X,\theta)$. We need to show that
    \[
    \lim_{j\to\infty} \int_X (\theta+\ddc\varphi^j)^n=\int_X (\theta+\ddc P_{\theta}[\varphi]_{\mathcal{I}})^n.
    \]
    This follows from \cite[Corollary~3.4]{DX21}, see also \cite[Corollary~7.1.2]{Xiabook}.
    
\end{proof}

\subsection{Properties of mixed volumes}

\begin{proposition}\label{prop:mono_vol}
    Let $S_1,\ldots,S_n$ be closed positive $(1,1)$-currents on $X$. Assume that for each $i=1,\ldots,n$,
    \begin{enumerate}
        \item $T_i\preceq_{\mathcal{I}} S_i$;
        \item $\{T_i\}=\{S_i\}$.
    \end{enumerate}
    Then
    \begin{equation}\label{eq:mono}
    \vol(T_1,\ldots,T_n)\leq \vol(S_1,\ldots,S_n).
    \end{equation}
\end{proposition}

\begin{proof}
    Let $\omega$ be a Kähler form on $X$. It suffices to show that for each $\epsilon>0$, we have
    \[
    \vol(T_1+\epsilon\omega,\ldots,T_n+\epsilon\omega)\leq \vol(S_1+\epsilon\omega,\ldots,S_n+\epsilon\omega).
    \]
    In particular, we reduce to the case where $\vol T_i>0$, $\vol S_i>0$ for each $i=1,\ldots,n$.

    In this case, \eqref{eq:mono} is a consequence of \cref{thm:mono2}.
\end{proof}

\begin{proposition}\label{prop:logvol}
    We have
    \[
    \vol(T_1,\ldots,T_n)\geq \prod_{i=1}^n(\vol T_i)^{1/n}.
    \]
\end{proposition}
\begin{proof}
    We may assume that $\vol T_i>0$ for each $i=1,\ldots,n$ since there is nothing to prove otherwise. In this case, we need to show that
    \[
    \int_X \left(\theta_1+\ddc P_{\theta_1}[\varphi_1]_{\mathcal{I}}\right)\wedge \cdots \wedge \left(\theta_n+\ddc P_{\theta_1}[\varphi_n]_{\mathcal{I}}\right)\geq \prod_{i=1}^n\left(\int_X \left(\theta_i+\ddc P_{\theta_i}[\varphi_i]_{\mathcal{I}} \right)^n \right)^{1/n}. 
    \]
    This is a special case of the main theorem of \cite{DDNL19log}.
\end{proof}

\begin{proposition}\label{prop:biminv}
    Let $\pi\colon Y\rightarrow X$ be a proper bimeromorphic morphism from a Kähler manifold $Y$ to $X$, then
    \[
    \vol(\pi^*T_1,\ldots,\pi^*T_n)=\vol(T_1,\ldots,T_n).
    \]
\end{proposition}
\begin{proof}
    As in the proof of \cref{prop:mono_vol}, we may easily reduce to the case where $\vol T_i>0$ for each $i=1,\ldots,n$. By \cite[Proposition~3.2.5]{Xiabook}, we know that if we write $T_i=\theta_i+\ddc\varphi_i$, then 
    \[
    \pi^*P_{\theta_i}[\varphi_i]_{\mathcal{I}}=P_{\pi^*\theta_i}[\pi^*\varphi_i]_{\mathcal{I}}. 
    \]
    In particular,
    \[
    \vol \pi^*T_i=\vol T_i>0.
    \]
    Our assertion follows from the obvious bimeromorphic invariance of the non-pluripolar product.
\end{proof}

\begin{lemma}\label{lma:perturbvol}
    Let $\omega$ be a Kähler form on $X$. Then there is a constant $C>0$ depending only on $X,\omega,\{\theta_1\},\ldots,\{\theta_n\}$ such that
    \[
    0\leq \vol(T_1+\epsilon\omega,\ldots, T_n+\epsilon\omega)-\vol(T_1,\ldots,T_n)\leq C\epsilon
    \]
    for any $\epsilon\in [0,1]$.
\end{lemma}
\begin{proof}
    By linearity, we can write
    \[
    \vol(T_1+\epsilon\omega,\ldots, T_n+\epsilon\omega)-\vol(T_1,\ldots,T_n)
    \]
    as a linear combination of the mixed volumes between the $T_i$'s and $\omega$ with coefficients $\epsilon^j$ for some $j\geq 1$. The mixed volumes are clearly bounded by a constant.
\end{proof}

\begin{proposition}\label{prop:strongcontnpm}
    Let $(T_i^j)_{j\in J}$ be nets of closed positive $(1,1)$-currents on $X$ for each $i=1,\ldots,n$. Assume that for each $i=1,\ldots,n$, we have
    \[
    T_i^j\implies T_i.
    \]
    Then
    \begin{equation}\label{eq:strongcontnpm}
    \lim_{j\in J}\int_X T_1^j\wedge \cdots \wedge T_n^j=\int_X T_1\wedge \cdots \wedge T_n,
    \end{equation}
    and
    \begin{equation}\label{eq:strongcontvol}
    \lim_{j\in J}\vol \left( T_1^j,\ldots, T_n^j\right)=\vol \left( T_1, \ldots, T_n\right).
    \end{equation}
\end{proposition}
Recall that $\implies$ is defined in \cref{def:stronconve}.
\begin{proof}
    Let $\omega$ be a Kähler form on $X$. For each $\epsilon>0$, we can find $j_0\in J$ so that for $j\geq j_0$, the following classes are Kähler:
    \[
    \{T_i\}+2^{-1}\epsilon\{\omega\}-\{T_i^j\},\quad i=1,\ldots,n.
    \]
    Take a Kähler form $\omega_i^j$ in the class $ \{T_i\}+\epsilon\{\omega\}-\{T_i^j\}$. Then observe that for $i=1,\ldots,n$,
    \[
    T_i^j+\omega_i^j\xrightarrow{d_S} T_i+\epsilon\omega.
    \]
    Since these currents are now in the same cohomology class, it follows from \cite[Theorem~4.2]{Xia21} (see also \cite[Theorem~6.2.1]{Xiabook}) that
    \begin{equation}\label{eq:mixedmassconv}
    \lim_{j\in J}\int_X (T_1^j+\omega_1^j)\wedge \cdots\wedge (T_n^j+\omega_n^j)=\int_X (T_1+\epsilon\omega)\wedge \cdots \wedge (T_n+\epsilon\omega).
    \end{equation}
    Note that we can find a constant $C>0$ independent of $j\geq j_0$ so that for any $j\geq j_0$, we have
    \begin{equation}\label{eq:errorTij}
    \begin{split}
    \int_X (T_1^j+\omega_1^j)\wedge \cdots\wedge (T_n^j+\omega_n^j)-\int_X T_1^j \wedge \cdots\wedge T_n^j\leq C\epsilon,\\ \int_X (T_1+\epsilon\omega)\wedge \cdots \wedge (T_n+\epsilon\omega)-\int_X T_1\wedge \cdots \wedge T_n\leq C\epsilon.
    \end{split}
    \end{equation}
    Hence \eqref{eq:strongcontnpm} follows.

    As for \eqref{eq:strongcontvol}, it suffices to replace \eqref{eq:errorTij} by \cref{lma:perturbvol}, and
    \eqref{eq:mixedmassconv} by
    \[
    \lim_{j\in J}\vol \left(T_1^j+\omega_1^j,\ldots,T_n^j+\omega_n^j\right)=\vol \left(T_1+\epsilon\omega,\ldots,T_n+\epsilon\omega\right),
    \]
    which follows from \cite[Theorem~4.2, Theorem~4.6]{Xia21} (see also
    \cite[Theorem~6.2.1, Theorem~6.2.3]{Xiabook}). 
\end{proof}

Next we establish a semicontinuity property of the mixed volumes.
\begin{theorem}\label{thm:usc_mix_vol}
   Let $(\varphi_i^j)_{j\in J}$ ($i=1,\ldots,n$) be nets in $\PSH(X,\theta_i)$. Assume that for each prime divisor $E$ over $X$, we have 
   \[
   \lim_{j\in J} \nu(\varphi_i^j,E)=\nu(\varphi_i,E),\quad i=1,\ldots,n.
   \]
   Then
   \[
   \varlimsup_{j\in J} \vol \left(\theta_1+\ddc \varphi_1^j, \dots ,\theta_n+\ddc \varphi_n^j \right) \leq  \vol\left(\theta_1+\ddc \varphi_1, \dots ,\theta_n+\ddc \varphi_n\right).
   \]
\end{theorem}
\begin{proof}
    \textbf{Step~1}. We first assume that $\vol (\theta_i+\ddc \varphi_i^j)>0$ and $\vol (\theta_i+\ddc \varphi_i)>0$ for all $i=1,\ldots,n$ and $j\in J$.

    Without loss of generality, we may assume that the $\varphi_i^j$'s and the $\varphi_i$'s are $\mathcal{I}$-model for all $i=1,\ldots,n$ and $j\in J$. Our assertion becomes
    \begin{equation}\label{eq:volineq_1}
    \varlimsup_{j\in J} \int_X \left(\theta_1+\ddc \varphi_1^j\right)\wedge \dots \wedge \left(\theta_n+\ddc \varphi_n^j \right) \leq  \int_X \left(\theta_1+\ddc \varphi_1\right) \wedge \dots \wedge \left(\theta_n+\ddc \varphi_n\right).
    \end{equation}
    For each $j\in J$, define
    \[
    \psi_i^j\coloneqq \sups_{\!\!\! k\geq j} \varphi_i^k,\quad i=1,\ldots,n.
    \]
    Observe that $\psi_i^j$ is $\mathcal{I}$-good thanks to \cref{prop:Igoodclosed}.
    It follows from \cite[Corollary~1.4.1]{Xiabook} and our assumption that
    \[
    \lim_{j\in J}\nu\left(\psi_i^j,E\right)=\nu\left(\varphi_i,E\right),\quad i=1,\ldots,n.
    \]
    For each $i=1,\ldots,n$, we define
    \[
    \psi_i=\inf_{j\in J}P_{\theta_i}[\psi_i^j].
    \]
    Due to \cite[Lemma~2.21]{DX22} (see also \cite[Proposition~3.2.12]{Xiabook}), $\psi_i$ is $\mathcal{I}$-model.
    Thanks to \cite[Proposition~3.1.10]{Xiabook}, we know
    \[
    \nu(\psi_i,E)=\nu(\varphi_i,E)
    \]
    for any $i=1,\ldots,n$ and any prime divisor $E$ over $X$. In other words, $\psi_i\sim_{\mathcal{I}} \varphi_i$ for $i=1,\ldots,n$. But both $\varphi_i$ and $\psi_i$ are $\mathcal{I}$-good, therefore, 
    \[
    \psi_i\sim_{P} \varphi_i,\quad i=1,\ldots,n.
    \]
    By \cref{thm:mono2}, we have
    \[
    \int_X (\theta_1+\ddc \psi_1)\wedge \cdots\wedge (\theta_n+\ddc \psi_n)=\int_X (\theta_1+\ddc \varphi_1)\wedge \cdots\wedge (\theta_n+\ddc \varphi_n).
    \]
    Next by \cref{thm:mono2} again, 
    \[
    \varlimsup_{j\in J} \int_X \left(\theta_1+\ddc \varphi_1^j\right)\wedge\dots\wedge  \left(\theta_n+\ddc \varphi_n^j \right)\leq \varlimsup_{j\in J} \int_X \left(\theta_1+\ddc \psi_1^j\right)\wedge \dots \wedge \left(\theta_n+\ddc \psi_n^j \right).
    \]
    On the other hand, due to \cite[Proposition~4.8]{DDNLmetric}, for each $i=1,\ldots,n$, we have
    \[
    \psi_i^j\xrightarrow{d_S}\psi_i.
    \]
    We conclude from \cref{prop:strongcontnpm} that
    \[
    \varlimsup_{j\in J} \int_X \left(\theta_1+\ddc \psi_1^j\right)\wedge \dots \wedge \left(\theta_n+\ddc \psi_n^j \right)=\int_X (\theta_1+\ddc \psi_1)\wedge \cdots\wedge (\theta_n+\ddc \psi_n).
    \]
    Putting these equations together, \eqref{eq:volineq_1} follows.

    \textbf{Step~2}. Next we handle the general case. 

    Fix a Kähler form $\omega$ on $X$. For any $\epsilon\in (0,1]$, from Step~1, we know that
    \[
   \varlimsup_{j\in J} \vol \left(\theta_1+\epsilon\omega+\ddc \varphi_1^j, \dots ,\theta_n+\epsilon\omega+\ddc \varphi_n^j \right) \leq  \vol\left(\theta_1+\epsilon\omega+\ddc \varphi_1, \dots ,\theta_n+\epsilon\omega+\ddc \varphi_n\right).
   \]
   Using \cref{lma:perturbvol}, we have
   \[
   \begin{aligned}
     &\varlimsup_{j\in J} \vol \left(\theta_1+\ddc \varphi_1^j, \dots ,\theta_n+\ddc \varphi_n^j \right)\\
     \leq & \varlimsup_{j\in J} \vol \left(\theta_1+\epsilon\omega+\ddc \varphi_1^j, \dots ,\theta_n+\epsilon\omega+\ddc \varphi_n^j \right)\\
     \leq &\vol\left(\theta_1+\epsilon\omega+\ddc \varphi_1, \dots ,\theta_n+\epsilon\omega+\ddc \varphi_n\right)\\
     \leq &\vol\left(\theta_1+\ddc \varphi_1, \dots ,\theta_n+\ddc \varphi_n\right)+C\epsilon.
    \end{aligned}
   \]
   But since $\epsilon$ is arbitrary, our assertion follows.
\end{proof}

\begin{lemma}\label{lma:Siudec_pullpush}
    Let $\pi\colon Y\rightarrow X$ be a proper bimeromorphic morphism from a Kähler manifold $Y$.
    Then for any non-divisorial closed positive $(1,1)$-current $T$ on $Y$, we have
    \[
    \pi^* \pi_* T=T+\sum_{i=1}^N c_i [E_i]
    \]
    for finitely many $\pi$-exceptional divisors $E_i$ and $c_i>0$.

    In particular, if $S$ is a closed positive $(1,1)$-current on $X$, we can find $E_i$ and $c_i$ as above so that
    \[
    \pi^*\langle S \rangle = \langle \pi^*S \rangle +\sum_{i=1}^N c_i [E_i].
    \]
\end{lemma} 

\begin{proof}
    Let $E$ be the exceptional locus of $\pi$. Then 
    \[
    T=\mathds{1}_{Y\setminus E}\pi^* \pi_* T.
    \]
    Therefore, 
    \[
    \pi^* \pi_* T-T=\mathds{1}_E\pi^* \pi_* T,
    \]
    which has the stated form, due to the support theorems, see \cite[Section~8]{Dembook2}.
\end{proof}

It turns out that the mixed volume depends only on the regular parts of the currents. 
\begin{theorem}\label{thm:vol_reg_parts}
    We have
    \[
    \vol \left(T_1,\ldots,T_n\right)=\vol\left(\Reg T_1,\ldots,\Reg T_n\right).
    \]
\end{theorem}
Recall that $\Reg$ is defined in \eqref{eq:Siudec}.
\begin{remark}
    In general, it is not true that the mixed volume depends only on the non-pluripolar parts of the currents. This even fails for the pure volume, see \cite[Example~6.10]{BBJ21} for an example.
\end{remark}

\begin{proof}
    \textbf{Step~1}. We first prove the assertion when $T_1=\cdots=T_n=T$ and $\vol T>0$. We want to show that
    \[
    \vol T=\vol \Reg T.
    \]

    We decompose $T$ as in \eqref{eq:Siudec}. 
    
    We first handle the case where the collection of the $E_i$'s is finite.
    In this case, it suffices to prove the following: If $S$ is a closed positive $(1,1)$-current on $X$, and $E$ is a prime divisor on $X$, then $\vol (S+[E])=\vol S$. For this purpose, we may assume that $S$ is a Kähler current. Take a quasi-equisingular approximation $(S_j)_j$ of $S$, note that $(S_j+[E])_j$ is a quasi-equisingular approximation of $S+[E]$ by \cite[Theorem~6.2.2, Corollary~7.1.2]{Xiabook}. Hence, we may finally assume that $S$ has analytic singularities. In this case the assertion is obvious.

    So we may assume that the index $i$ runs over all positive integers. From the previous argument, we know that for any $N\geq 0$,
    \[
    \vol T=\vol \left(T-\sum_{i=1}^N c_i[E_i]\right).
    \]
    Thanks to \cref{prop:strongcontnpm} and \cite[Theorem~6.2.2]{Xiabook}, it suffices to show that
    \begin{equation}\label{eq:finitediv_convinf}
    \sum_{i=1}^N c_i[E_i]\implies \sum_{i=1}^{\infty} c_i[E_i]
    \end{equation}
    as $N\to\infty$.

    Fix a Kähler form $\omega$ on $X$. 
    We can find $N_0>0$ so that for any $N\geq N_0$, the class of
    \[
    \omega+ \sum_{i=N+1}^{\infty}c_i[E_i]
    \]
    is Kähler. Take a Kähler form $\omega_N$ in this class.
    Then the currents 
    \[
    \sum_{i=1}^N c_i [E_i]+\omega_N,\quad \sum_{i=1}^{\infty} c_i [E_i]+\omega
    \]
    all lie in the same cohomology class.
    So our problem is reduced to
    \[
    \sum_{i=1}^N c_i [E_i]+\omega_N\xrightarrow{d_S}\sum_{i=1}^{\infty} c_i [E_i]+\omega.
    \]
    In fact, it suffices to show the convergence of the non-pluripolar masses, due to \cite[Corollary~6.2.5]{Xiabook}. In other words, we need to show that
    \[
    \lim_{N\to\infty}\int_X \omega_N^n=\int_X \omega^n,
    \]
    which follows from the convergence $\{\omega_N\}\to \{\omega\}$.

    \textbf{Step~2}. We handle the general case. Fix a Kähler form $\omega$ on $X$, by Step~1, for any $d_1,\ldots,d_n>0$ and $\epsilon>0$, we have
    \[
    \vol\left( \sum_{i=1}^n d_i (T_i+\epsilon\omega) \right)=\vol\left( \sum_{i=1}^n d_i (\Reg T_i+\epsilon\omega) \right).
    \]
    Since both sides are polynomials in $d_1,\ldots,d_n$, we conclude that 
    \[
    \vol (T_1+\epsilon\omega,\ldots,T_n+\epsilon\omega)=\vol (\Reg T_1+\epsilon\omega,\ldots,\Reg T_n+\epsilon\omega).
    \]
    Letting $\epsilon\to 0+$, we conclude our assertion.
\end{proof}

\begin{corollary}\label{prop:vol_inv_push}
    Let $\pi\colon X\rightarrow Z$ be a proper bimeromorphic morphism from $X$ to a Kähler manifold $Z$. Then
    \begin{equation}\label{eq:volpushinc}
    \vol(T_1,\ldots,T_n)=\vol(\pi_*T_1,\ldots,\pi_*T_n).
    \end{equation}
\end{corollary}
\begin{proof}

    Observe that we may assume that $T_i=\Reg T_i$ for all $i=1,\ldots,n$. In fact, clearly the pushforward of the divisorial part of $T_i$ is divisorial as well, hence by \cref{thm:vol_reg_parts}, they do not contribute to the volumes.
    
    Now by \cref{prop:biminv}, it remains to show that
    \[
    \vol(T_1,\ldots,T_n)=\vol(\pi^*\pi_*T_1,\ldots,\pi^*\pi_*T_n).
    \]
    By \cref{lma:Siudec_pullpush}, the difference $\pi^*\pi_*T_i-T_i$ is divisorial, hence our desired equality follows from \cref{thm:vol_reg_parts}.
\end{proof}

A particular corollary of \cref{prop:vol_inv_push} will be useful later.
\begin{corollary}\label{cor:Igoodpush}
     Let $\pi\colon X\rightarrow Z$ be a proper bimeromorphic morphism from $X$ to a Kähler manifold $Z$. Assume that $T$ is an $\mathcal{I}$-good closed positive $(1,1)$-current on $X$, then so is $\pi_*T$.
\end{corollary}
\begin{proof}
    We may assume that $\int_X T^n>0$. Then by \cref{prop:vol_inv_push}, 
    \[
    \vol \pi_*T=\vol T>0
    \]
    as well.
    Since $T$ is $\mathcal{I}$-good, we have
    \[
    \vol T=\int_X T^n.
    \]
    But $\int_X T^n=\int_Z (\pi_*T)^n$, so
    \[
    \vol \pi_*T=\int_Z (\pi_*T)^n>0.
    \]
    It follows that $\pi_*T$ is $\mathcal{I}$-good.
\end{proof}

\begin{lemma}\label{lma:pushIpo}
    Let $\pi\colon X\rightarrow Z$ be a proper bimeromorphic morphism from $X$ to a Kähler manifold $Z$. Consider non-divisorial closed positive $(1,1)$ currents $T,S$ on $X$ in the same cohomology class. Assume that $T\preceq_{\mathcal{I}} S$, then $\pi_*T\preceq_{\mathcal{I}} \pi_*S$.
\end{lemma}
\begin{proof}
    We may assume that $\pi$ is a modification thanks to Hironaka's Chow lemma \cite[Corollary~2]{Hir75} and \cref{lma:Ipartorderinv}.

    By \cref{lma:Siudec_pullpush},
    \[
    \pi^*\pi_*T=T+\sum_{i=1}^N c_i[E_i],
    \]
    where $c_i>0$ and the $E_i$'s are $\pi$-exceptional divisors. It follows that 
    \[
    T+\sum_{i=1}^N c_i[E_i]\preceq_{\mathcal{I}}S+\sum_{i=1}^N c_i[E_i].
    \]
    Replacing $T$ and $S$ by $T+\sum_{i=1}^N c_i[E_i]$ and $S+\sum_{i=1}^N c_i[E_i]$ respectively, we may assume that $T=\pi^*\pi_*T$. In particular, $S$ and $\pi^*\pi_*S$ lie in the same cohomology class, and hence $S=\pi^*\pi_*S$ (\emph{c.f.} the proof of \cref{thm:bdiv_current_corr} below). Our assertion then follows from \cref{lma:Ipartorderinv}.
\end{proof}

\section{Transcendental b-divisors}

Let $X$ be a connected compact Kähler manifold of dimension $n$.

\subsection{The definitions}
The b-divisors defined in this section are sometimes known as b-divisor classes. We always omit the word \emph{classes} to save space.
\begin{definition}
    A \emph{(Weil) b-divisor} $\mathbb{D}$ over $X$ is an assignment $(\mathbb{D}_{\pi})_{\pi\colon Y\rightarrow X}$, where $\pi\colon Y\rightarrow X$ runs over all modifications of $X$ such that
    \begin{enumerate}
        \item $\mathbb{D}_{\pi}\in \mathrm{H}^{1,1}(Y,\mathbb{R})$;
        \item  The classes are compatible under push-forwards: If $\pi'\colon Z\rightarrow X$ and $\pi\colon Y\rightarrow X$ are both in $\Modif(X)$ and $\pi'$ dominates $\pi$ through $g\colon Z\rightarrow Y$ (namely, $g$ makes the diagram \eqref{eq:domi} commutative), then $g_*\mathbb{D}_{\pi'}=\mathbb{D}_{\pi}$.
    \end{enumerate}
    We also write $\mathbb{D}_Y=\mathbb{D}_{\pi}$ if there is no risk of confusion.

    Given two Weil b-divisors $\mathbb{D}$ and $\mathbb{D}'$ over $X$, we say $\mathbb{D}\leq \mathbb{D}'$ if for each $\pi\in \Modif(X)$, we have $\mathbb{D}_{\pi}\leq \mathbb{D}'_{\pi}$. Recall that by definition, this means the class $\mathbb{D}'_{\pi}-\mathbb{D}_{\pi}$ is pseudoeffective.
\end{definition}
The class $\mathbb{D}_X$ is called the \emph{root} of $\mathbb{D}$.
The set of Weil b-divisors over $X$ has the obvious structure of real vector spaces.

\begin{definition}\label{def:volb}
    The \emph{volume} of a Weil b-divisor $\mathbb{D}$ over $X$ is 
    \[
        \vol \mathbb{D}\coloneqq \lim_{\pi\colon Y\rightarrow X} \vol \mathbb{D}_Y.
    \]
    The right-hand side is a decreasing net due to \cref{prop:cones_modif}, hence the limit always exists.

    We say $\mathbb{D}$ is \emph{big} if $\vol \mathbb{D}>0$.
\end{definition}

\begin{lemma}\label{lma:volbdivnet}
    Let $(\mathbb{D}_i)_{i\in I}$ be a net of b-divisors converging to $\mathbb{D}$. Then
    \begin{equation}\label{eq:vollimsup}
    \varlimsup_{i\in I}\vol \mathbb{D}_i\leq \vol \mathbb{D}.
    \end{equation}
    If the net is decreasing, then
    \[
    \lim_{i\in I}\vol \mathbb{D}_i=\vol \mathbb{D}.
    \]
\end{lemma}
Here we say $(\mathbb{D}_i)_{i\in I}$ converges to $\mathbb{D}$ if for any modification $\pi\colon Y\rightarrow X$, we have $\mathbb{D}_{i,Y}\to \mathbb{D}_Y$ with respect to the Euclidean topology.

In general, we cannot expect equality in \eqref{eq:vollimsup}, as shown by \cite[Example~3.3]{DF20}. 
\begin{proof}
    Let $\pi\colon Y\rightarrow X$ be a modification. Then
    \[
    \vol \mathbb{D}_Y=\lim_{i\in I} \vol \mathbb{D}_{i,Y}\geq \varlimsup_{i\in I} \vol \mathbb{D}_i.
    \]
    The inequality \eqref{eq:vollimsup} follows.
    As for the decreasing case, it suffices to observe that both sides of \eqref{eq:vollimsup} can be written as
    \[
    \inf_i \inf_{\pi\colon Y\rightarrow X} \vol \mathbb{D}_{i,Y}.
    \]
\end{proof}

\begin{definition}
    A \emph{Cartier b-divisor} $\mathbb{D}$ over $X$ is a Weil b-divisor $\mathbb{D}$ over $X$ such that there exists a modification $\pi\colon Y\rightarrow X$ and a class $\alpha_Y\in \mathrm{H}^{1,1}(Y,\mathbb{R})$ so that for each $\pi'\colon Z\rightarrow X$ dominating $\pi$, the class $\mathbb{D}_Z$ is the pull-back of $\alpha_Y$. Any such $(\pi,\alpha_Y)$ is called a \emph{realization} of $\mathbb{D}$.
\end{definition}
By abuse of language, we also say $(Y,\alpha_Y)$ is a realization of $\mathbb{D}$.
The realization is not unique in general.

\begin{definition}\label{def:nefCar}
    A Cartier b-divisor $\mathbb{D}$ over $X$ is \emph{nef} if there exists a realization $(\pi\colon Y\rightarrow X, \alpha_Y)$ of $\mathbb{D}$ such that $\alpha_Y$ is nef.
\end{definition}

\begin{definition}\label{def:nefb}
    A Weil b-divisor $\mathbb{D}$ over $X$ is \emph{nef} if there is a net of nef Cartier b-divisors $(\mathbb{D}_i)_i$ over $X$ 
    converging to $\mathbb{D}$.
\end{definition}
In other words, for each modification $\pi\colon Y\rightarrow X$, we have $\mathbb{D}_{i,Y}\to \mathbb{D}_Y$.

Note that thanks to \cref{prop:cones_modif}, each $\mathbb{D}_Y$ is necessarily modified nef, but it is not nef in general. 

\emph{A priori}, for a Cartier b-divisor, nefness could mean two different things, either  defined by \cref{def:nefCar} or by \cref{def:nefb}. We will show in \cref{cor:nefbCequiv} that they are actually equivalent. Before that, by a nef Cartier b-divisor, we always mean in the sense of \cref{def:nefCar}.

Our definition \cref{def:nefb} amounts defining the set of Weil b-divisors as the closure of the set of Cartier b-divisors in $\varprojlim_{\pi}\mathrm{H}^{1,1}(Y,\mathbb{R})$ with respect to the projective limit topology. In particular, the limit of a converging net of nef b-divisors is still nef.

\subsection{The b-divisors of currents}\label{subsec:b-div_curr}

Let $T$ be a closed positive $(1,1)$-current on $X$.

Given any modification $\pi\colon Y\rightarrow X$, we define
\begin{equation}
\mathbb{D}(T)_Y\coloneqq \left\{ \Reg \pi^*T \right\}\in \mathrm{H}^{1,1}(Y,\mathbb{R}).
\end{equation}
The b-divisor $\mathbb{D}(T)$ was firstly explicitly introduced in \cite{XiaPPT} in 2020. The paper received very little attention and the same object was re-introduced in \cite{BBGHdJ21} and \cite{Tru24} later on.

We observe that if $T'$ is another closed positive $(1,1)$-current on $X$ and $\lambda\geq 0$, then
\[
\mathbb{D}(T+T')=\mathbb{D}(T)+\mathbb{D}(T'),\quad \mathbb{D}(\lambda T)=\lambda \mathbb{D}(T).
\]
We shall use these identities implicitly in the sequel.

Note that when $T$ has analytic singularities, $\mathbb{D}(T)$ is Cartier.

\begin{lemma}\label{lma:DT}
    Let $T$ be a closed positive $(1,1)$-current on $X$. Then $\mathbb{D}(T)$ is nef. Moreover,
    \begin{equation}\label{eq:volDT}
    \vol T=\vol \mathbb{D}(T).
    \end{equation}
\end{lemma}

\begin{proof}
    Let $\omega$ be a Kähler form on $X$. Then $\mathbb{D}(\omega)$ is the Cartier b-divisor realized by $(X,\{\omega\})$.
    We could always approximate $\mathbb{D}(T)$ by $\mathbb{D}(T+\epsilon\omega)=\mathbb{D}(T)+\epsilon \mathbb{D}(\omega)$. 
    Moreover, we can find a constant $C>0$ so that
    \begin{equation}\label{eq:volDTp}
    0\leq \vol \left( \mathbb{D}(T)+ \epsilon\mathbb{D}(\omega) \right)-\vol \mathbb{D}(T)\leq C\epsilon.
    \end{equation}
    Hence we may assume that $T$ is a Kähler current.    

    Next, we take a closed smooth real $(1,1)$-form $\theta$ cohomologous to $T$ and write $T=\theta_{\varphi}$ for some $\varphi\in \PSH(X,\theta)$. Let $(\varphi_j)_j$ be a quasi-equisingular approximation of $\varphi$ in $\PSH(X,\theta)$. Then it is easy to see that $\mathbb{D}(\theta+\ddc \varphi_j)\to \mathbb{D}(\theta+\ddc \varphi)$.
    See \cite[Theorem~9.6]{Xia24} Step~2 for the details. As a consequence,
    \[
    \vol \mathbb{D}(\theta+\ddc \varphi_j)\to \vol \mathbb{D}(\theta+\ddc \varphi),
    \]
    thanks to \cref{lma:volbdivnet}.

    So we may assume that $T$ has analytic singularities. Let $\pi\colon Y\rightarrow X$ be a modification so that
    \[
    \pi^*T=[D]+R,
    \]
    where $D$ is an effective $\mathbb{Q}$-divisor on $Y$ and $R$ is a closed positive $(1,1)$-current with locally bounded potentials. Then $\mathbb{D}(T)$ is the nef Cartier b-divisor realized by $(\pi,\{R\})$. Note that \eqref{eq:volDT} is obvious in this case.
\end{proof}

\begin{remark}
There is a different possibility: Replace $\Reg$ by the non-pluripolar part. Given $T$ as above, we define
\[
\mathbb{D}'(T)_{\pi} \coloneqq  [\langle \pi^* T \rangle ].
\]
But thanks to \cref{lma:Siudec_pullpush}, we have
\[
\mathbb{D}'(T)=\mathbb{D}(\langle T \rangle),
\]    
so there is no new information.
\end{remark}

Conversely, we want to realize nef b-divisors as $\mathbb{D}(T)$. We first prove a continuity result.
\begin{proposition}\label{prop:dscontbdiv}
Let $(T_i)_{i\in I}$ be a net of closed positive $(1,1)$-currents on $X$ and $T$ be a closed positive $(1,1)$-current on $X$. Assume that $T_i\implies T$, then 
\[
    \mathbb{D}(T_i)\to \mathbb{D}(T).
\]
\end{proposition}
Recall that $\implies$ is defined in \cref{def:stronconve}.
\begin{proof}
    When the cohomology classes $\{T_i\}$ and $\{T\}$ are all the same, the proof is the same as that in the algebraic case, which we omit. See \cite[Theorem~9.8]{Xia24}.

    In general, fix a Kähler form $\omega$ on $X$. Then we can find $i_0\in I$ so that for $i\geq i_0$, the class
    \[
    \omega+\{T\}-\{T_i\}
    \]
    is Kähler, and we can find a Kähler form $\omega_i$ in this class. It follows that
    \[
    T_i+\omega_i,\quad T+\omega
    \]
    are all in the same cohomology class, and hence
    \[
    \mathbb{D}(T_i+\omega_i)\to \mathbb{D}(T+\omega).
    \]
    But clearly
    \[
    \mathbb{D}(\omega_i)\to \mathbb{D}(\omega),
    \]
    so our assertion follows.
\end{proof}

\begin{theorem}\label{thm:bigandnefreal}
    Each big and nef b-divisor $\mathbb{D}$ over $X$ can be realized as $\mathbb{D}(T)$ for some $T\in \mathbb{D}_X$. Furthermore, we may always assume that $T$ is $\mathcal{I}$-good.
\end{theorem}
Note that $T$ is not unique. The current $T$ is necessarily non-divisorial. 

\begin{proof}
    Fix a big and nef b-divisor $\mathbb{D}$ over $X$. 

    For each $\pi\colon Y\rightarrow X$, we take a current with minimal singularities $T_Y$ in $\mathbb{D}_Y$. 
    We claim that $\mathbb{D}(\pi_*T_Y)$ coincides with $\mathbb{D}$ up to the level of $Y$: For any modification $\pi'\colon Z\rightarrow X$ dominated by $\pi$ through a morphism $g\colon Y\rightarrow Z$, we have
    \[
    \mathbb{D}_Z=\mathbb{D}(\pi_*T_Y)_Z.
    \]
    The notations are summarized in the following commutative diagram:
    \begin{equation}\label{eq:domi2}
    \begin{tikzcd}
Y \arrow[rr,"g"] \arrow[rd, "\pi"'] &   & Z \arrow[ld, "\pi'"] \\
                                 & X. &                    
\end{tikzcd}
    \end{equation}
    After unfolding the definitions, this means
    \[
        \Reg(\pi'^*\pi_*T_Y)\in \mathbb{D}_Z.
    \]
    Note that
    \[
        \Reg(\pi'^*\pi_*T_Y)=\Reg(\pi'^* \pi'_* g_*T_Y).
    \]
    Due to \cref{prop:cones_modif}, we know that $\mathbb{D}_Y$ is modified nef and big. In particular, $T_Y$ is non-divisorial, hence so is $g_*T_Y$ by \cref{lma:ndvpush}. It follows from \cref{lma:Siudec_pullpush} that
    \[
    \Reg(\pi'^* \pi'_* g_*T_Y)=\Reg(g_*T_Y)=g_*T_Y\in \mathbb{D}_Z.
    \]
    Note that
    \begin{equation}\label{eq:vol_unif_bound}
    \vol T_Y\geq \vol \mathbb{D}>0.
    \end{equation}
    
    Next we claim that the $P$-singularity types of the net $(\pi_*T_Y)_Y$ is decreasing.

    To see this, let us fix a diagram as \eqref{eq:domi2}. We need to show that 
    \[
    \pi_*T_Y\preceq_P \pi'_*T_Z.
    \]
    Since $T_Z$ has minimal singularities, it is clear that $g_*T_Y\preceq_{\mathcal{I}} T_Z$. In particular, \cref{lma:pushIpo} guarantees that $\pi_*T_Y\preceq_{\mathcal{I}} \pi'_*T_Z$. But thanks to \cref{cor:Igoodpush}, both $\pi_*T_Y$ and $\pi'_*T_Z$ are $\mathcal{I}$-good, so there is no difference between the $P$-partial order and the $\mathcal{I}$-partial order in this case. Our assertion follows.

    Next we claim that the net $(\pi_*T_Y)_Y$ has a $d_S$-limit.
    
    To see, let us take a smooth closed real $(1,1)$-form $\theta$ in $\alpha$ and write $\pi_*T_Y$ as $\theta+\ddc \varphi_Y$ for some $\theta$-psh function $\varphi_Y$. It suffices to show that the net $(\varphi_Y)_Y$ has a $d_S$-limit. But as we recalled earlier, the $d_S$-pseudometric gives $0$-distance to $P$-equivalent potentials, so it suffices to show that the decreasing net $(P_{\theta}[\varphi_Y])_Y$ has a $d_S$-limit. This follows from \eqref{eq:vol_unif_bound} and \cite[Corollary~6.2.6]{Xiabook} (which is just a reformulation of the completeness theorem of Darvas--Di Nezza--Lu \cite[Theorem~1.1]{DDNLmetric}).

    Take a closed positive $(1,1)$-current $T\in \mathbb{D}_X$ such that 
    \[
    \pi_*T_Y\xrightarrow{d_S}T.
    \]
    It follows from \cref{prop:dscontbdiv} that
    \[
    \mathbb{D}(\pi_*T_Y)\to \mathbb{D}(T).
    \]
    Therefore, we conclude that
    \[
    \mathbb{D}(T)=\mathbb{D}.
    \]
    Thanks to \cref{lma:DT}, $\vol T>0$. Write $T=\theta+\ddc \varphi$ for some $\varphi\in \PSH(X,\theta)$, then 
    \[
    T'\coloneqq \theta+\ddc P_{\theta}[\varphi]_{\mathcal{I}}
    \]
    is $\mathcal{I}$-good, non-divisorial and $\mathbb{D}(T')=\mathbb{D}(T)$.
\end{proof}

Let $\alpha$ be a modified nef class on $X$. We write $\mathcal{G}(\alpha)$ for the set of closed positive $(1,1)$-currents $T$ on $X$ with $T=\Reg T\in \alpha$ and $\vol T>0$.
\begin{theorem}\label{thm:bdiv_current_corr}
    There is a natural bijection from $\mathcal{G}(\alpha)/\sim_{\mathcal{I}}$ to the set of big and nef b-divisors $\mathbb{D}$ over $X$ with $\mathbb{D}_X=\alpha$.
\end{theorem}
\begin{proof}

    Given $T\in \mathcal{G}(\alpha)$, we associate the b-divisor $\mathbb{D}(T)$. It is big and nef due to \cref{lma:DT}. This map clearly descends to $\mathcal{G}(\alpha)/\sim_{\mathcal{I}}$. 

    This map is surjective by \cref{thm:bigandnefreal}. Now we show that it is injective. Let $T,T'\in \mathcal{G}(\alpha)$. Assume that $\mathbb{D}(T)=\mathbb{D}(T')$, we want to show that $T\sim_{\mathcal{I}} T'$.

    Let $E$ be a prime divisor over $X$, it suffices to show that
    \begin{equation}\label{eq:nuTE}
    \nu(T,E)=\nu(T',E).
    \end{equation}
    We may assume that $E$ is not a prime divisor on $X$, as otherwise both sides vanish.
    
    Choose a sequence of blow-ups with smooth \emph{connected} centers 
    \[
    Y\coloneqq X_k\rightarrow X_{k-1}\rightarrow \cdots \rightarrow X_0\coloneqq X
    \]
    so that $E$ is a prime divisor on $Y$, exceptional with respect to $X_k\rightarrow X_{k-1}$. Denote the composition by $\pi\colon Y\rightarrow X$.
    Thanks to \cref{prop:blowup_cone},
    \[
    \mathrm{H}^{1,1}(X_k,\mathbb{R})=\mathrm{H}^{1,1}(X_{k-1},\mathbb{R})\oplus \mathbb{R}\{E_k\},
    \]
    where $E_k=E$ is the exceptional divisor of $X_k\rightarrow X_{k-1}$.  
    
    By induction, 
    \[
    \mathrm{H}^{1,1}(Y,\mathbb{R})=\mathrm{H}^{1,1}(X,\mathbb{R})\oplus \bigoplus_{i=1}^{k} \mathbb{R} \{E_i\},
    \]
    where $E_i$ is the exceptional divisor of $X_i\rightarrow X_{i-1}$. Now by \cref{lma:Siudec_pullpush},
    \begin{equation}\label{eq:RegpistarT}
    \Reg \pi^*T=\pi^*T-\sum_{i=1}^k \nu(T,E_i)[E_i].
    \end{equation}
    In particular, the cohomology class of $\Reg \pi^*T$ determines $\nu(T,E)$. Hence, \eqref{eq:nuTE} follows.
\end{proof}

\begin{corollary}
    The set of nef b-divisors over $X$ can be naturally identified with 
    \[
    \varprojlim_{\omega} \left( \mathcal{G}(\alpha+\omega)/\sim_{\mathcal{I}} \right),
    \]
    where $\omega$ runs over the directed set of Kähler forms on $X$ (with respect to the partial order of reverse domination), and given two Kähler forms $\omega\leq \omega'$ the transition map 
    \[
    \mathcal{G}(\alpha+\omega)/\sim_{\mathcal{I}}\rightarrow \mathcal{G}(\alpha+\omega')/\sim_{\mathcal{I}}
    \]
    is induced by the map $\mathcal{G}(\alpha+\omega)\rightarrow \mathcal{G}(\alpha+\omega')$ sending $T$ to $T+\omega'-\omega$.
\end{corollary}

\begin{corollary}\label{cor:nefbCequiv}
    Let $\mathbb{D}$ be a Cartier b-divisor over $X$. Then $\mathbb{D}$ is nef in the sense of \cref{def:nefCar} if and only if it is nef in the sense of \cref{def:nefb}.
\end{corollary}
This result is the transcendental version of  \cite[Theorem~2.8]{DF20}.
\begin{proof}
    We only handle the non-trivial implication. Assume that  $\mathbb{D}$ is nef in the sense of \cref{def:nefb}. We want to show that $\mathbb{D}$ is nef in the sense of \cref{def:nefCar}. We may clearly assume that $\mathbb{D}$ is big. Take a non-divisorial closed positive $(1,1)$-current $T$ on $X$ such that $\mathbb{D}=\mathbb{D}(T)$. 

    Without loss of generality, we may also assume that $\mathbb{D}$ is realized by $(X,\alpha)$ for some cohomology class $\alpha\in \mathrm{H}^{1,1}(X,\mathbb{R})$. Now $\mathbb{D}=\mathbb{D}(T)$ means that for each modification $\pi\colon Y\rightarrow X$, the current $\pi^*T$ is non-divisorial. In particular, $T$ has vanishing generic Lelong number along each prime divisor over $X$, see \eqref{eq:RegpistarT}. That means, $T$ has vanishing Lelong number everywhere. It follows that $\alpha=\{T\}$ is nef. 
\end{proof}

\begin{corollary}\label{cor:Dorder_Iorder}
    Let $T$ and $T'$ be non-divisorial closed positive $(1,1)$-currents on $X$. Suppose that $\{T\}=\{T'\}$, then the following are equivalent:
    \begin{enumerate}
        \item $\mathbb{D}(T)\leq \mathbb{D}(T')$;
        \item $T\preceq_{\mathcal{I}}T'$.
    \end{enumerate}
\end{corollary}
\begin{proof}
    This follows from \eqref{eq:RegpistarT}.
\end{proof}

In particular, we obtain the transcendental analogue of \cite[Theorem~A]{DF20}.
\begin{corollary}\label{cor:nef_Cartierseq}
    Let $\mathbb{D}$ be a nef b-divisor over $X$. Then there is a decreasing sequence of nef and big Cartier b-divisors $\mathbb{D}_i$ over $X$ with limit $\mathbb{D}$.
\end{corollary}
\begin{proof}
    Take a Kähler form $\omega$ on $X$. By \cref{thm:bigandnefreal}, for each $i>0$, we can find a non-divisorial Kähler current $T_i\in \mathbb{D}_X+i^{-1}\{\omega\}$ such that
    \[
    \mathbb{D}(T_i)=\mathbb{D}+i^{-1}\mathbb{D}(\omega).
    \]
    We observe that 
    \[
    T_{i+1}\sim_{\mathcal{I}} T_{i}.
    \]
    This follows from applying \cref{cor:Dorder_Iorder} to $T_i$ and $T_{i+1}+(i^{-1}-(i+1)^{-1})\omega$.
    Let $(T_i^j)_j$ be quasi-equisingular approximations of $T_i$ such that
    \begin{enumerate}
        \item 
        $T_i^j$ is a Kähler current in $\mathbb{D}_X+i^{-1}\{\omega\}$ for $j\geq j_0(i)$, and
        \item the singularity types of $(T_i^j)_i$ is constant.
    \end{enumerate}
    Note that (2) is possible by the using the Bergman kernel construction of the quasi-equisingular approximations. 
    
    It suffices to take $\mathbb{D}_i=\mathbb{D}(T_i^{j_i})$, where $j_i$ is a strictly increasing sequence of positive integers with $j_i\geq j_0(i)$.
\end{proof}

\section{The intersection theory}\label{sec:int_b_div}
Let $X$ be a connected compact Kähler manifold of dimension $n$. We will define the intersection numbers of nef b-divisors and show that they satisfy the same properties as their algebraic analogues, \emph{c.f.} \cite[Theorem~3.2]{DF20}. 

\begin{definition}\label{def:DF_ext}
    Let $\mathbb{D}_1,\ldots,\mathbb{D}_n$ be big and nef b-divisors over $X$. Then we define their intersection as
    \[
    \left( \mathbb{D}_1,\ldots,\mathbb{D}_n \right)\coloneqq \vol (T_1,\ldots,T_n),
    \]
    where $T_1,\ldots,T_n$ are closed positive $(1,1)$-currents in $\mathbb{D}_{1,X},\ldots,\mathbb{D}_{n,X}$ respectively such that $\mathbb{D}(T_i)=\mathbb{D}_{i}$.

    In general, if the $\mathbb{D}_i$'s are only nef, we define
    \[
    \left( \mathbb{D}_1,\ldots,\mathbb{D}_n \right)\coloneqq \lim_{\epsilon \to 0+}\left( \mathbb{D}_1+\epsilon \mathbb{D}(\omega),\ldots,\mathbb{D}_n+\epsilon \mathbb{D}(\omega) \right),
    \]
    where $\omega$ is a Kähler form on $X$.
\end{definition}
The definition makes sense thanks to \cref{thm:bigandnefreal}. It does not depend on the choices of $T_1,\ldots,T_n$ since they are uniquely defined up to $\mathcal{I}$-equivalence, as proved in \cref{thm:bdiv_current_corr}.

When $\mathbb{D}_1,\ldots,\mathbb{D}_n$ are big and nef, the two definitions coincide as follows from \cref{lma:perturbDF} below.

We first note that even when the $T_i$'s have vanishing volumes, the two intersection products still agree.
\begin{proposition}
    Let $T_1,\ldots,T_n$ be a closed positive $(1,1)$-currents on $X$. Then
    \[
    \left( \mathbb{D}(T_1),\ldots,\mathbb{D}(T_n) \right)= \vol (T_1,\ldots,T_n).
    \]
\end{proposition}
This is a trivial consequence of the definitions. 

\begin{proposition}\label{prop:DFlinear}
    The product in \cref{def:DF_ext} is symmetric and multi-$\mathbb{R}_{\geq 0}$-linear. 
\end{proposition}
\begin{proof}
    The multi-linearity follows immediately from \cref{prop:multilinear}. The symmetry is immediate.
\end{proof}

\begin{proposition}\label{prop:DFmono}
    The product in \cref{def:DF_ext} is monotonically increasing in each variable.
\end{proposition}
\begin{proof}
    Let $\mathbb{D}_1,\ldots,\mathbb{D}_n$ and $\mathbb{D}'$ be nef b-divisors over $X$ so that $\mathbb{D}_1\leq \mathbb{D}'$. We want to show that
    \[
    \left(\mathbb{D}_1,\ldots,\mathbb{D}_n  \right)\leq \left(\mathbb{D}',\mathbb{D}_2,\ldots,\mathbb{D}_n  \right).
    \]
    We can easily reduce to the case where $\mathbb{D}_1,\ldots,\mathbb{D}_n$ and $\mathbb{D}'$ are all big. In this case, take $\mathcal{I}$-good non-divisorial closed positive  $(1,1)$-currents $T_1,\ldots,T_n$ and $T'$ so that $\mathbb{D}(T_i)=\mathbb{D}_i$ for all $i=1,\ldots,n$ and $\mathbb{D}(T')=\mathbb{D}'$.
    Furthermore, we may assume that the $T_i$'s and $T'$ are Kähler currents by the perturbation argument.
    
    Let $(T_{i}^j)_j$ be a quasi-equisingular approximation of $T_i$ for $i=2,\ldots,n$. It follows from \cref{prop:strongcontnpm} that
    \[
    \int_X T_1\wedge \cdots \wedge T_n=\lim_{j\to\infty} \int_X T_1\wedge T_2^j\wedge\cdots \wedge T_n^j.
    \]
    It suffices to show that for all $j\geq 1$,
    \[
    \int_X T_1\wedge T_2^j\wedge\cdots \wedge T_n^j\leq \int_X T'\wedge T_2^j\wedge\cdots \wedge T_n^j.
    \]
    Therefore, we have reduced to the case where $T_2,\ldots,T_n$ have analytic singularities. After a resolution, we may assume that they have log singularities along $\mathbb{Q}$-divisors. By \cref{thm:vol_reg_parts}, we can further reduce to the case where $T_2,\ldots,T_n$ have bounded local potentials. Perturbing $T_2,\ldots,T_n$ by a Kähler form, we may further assume that $\{T_2\},\ldots, \{T_n\}$ are Kähler classes. By \cref{prop:mono_vol}, we finally reduce to the case where $T_2,\ldots,T_n$ are Kähler forms. In this case, our assertion is obvious.
\end{proof}

\begin{lemma}\label{lma:perturbDF}
    Let $\omega$ be a Kähler form on $X$. Fix a compact set $K\subseteq \mathrm{H}^{1,1}(X,\mathbb{R})$.
    Let $\mathbb{D}_1,\ldots,\mathbb{D}_n$ be nef b-divisors over $X$ such that $\mathbb{D}_{i,X}\in K$ for each $i=1,\ldots,n$. Then there is a constant $C$ depending only on $X,K,\{\omega\}$ such that for any $\epsilon\in [0,1]$, we have
    \[
    0\leq \left( \mathbb{D}_1+\epsilon \mathbb{D}(\omega),\ldots,\mathbb{D}_n+\epsilon \mathbb{D}(\omega) \right)-\left( \mathbb{D}_1,\ldots,\mathbb{D}_n \right)\leq C\epsilon.
    \]
\end{lemma}
\begin{proof}
    This is a simple consequence of the linearity \cref{prop:DFlinear}.
\end{proof}

We first make a consistency check.
\begin{proposition}
    Suppose that $\mathbb{D}$ is a nef b-divisor over $X$, then
    \[
    \left( \mathbb{D},\ldots,\mathbb{D} \right)=\vol \mathbb{D}.
    \]
\end{proposition}
\begin{proof}
    Using \cref{lma:perturbDF} and \eqref{eq:volDTp}, we may easily reduce to the case where $\mathbb{D}$ is nef and big. In this case, take a non-divisorial closed positive $(1,1)$-current $T$ in $\mathbb{D}_X$ such that $\mathbb{D}(T)=\mathbb{D}$. Then we need to show that
    \[
    \vol \mathbb{D}=\vol T,
    \]
    which is proved in \cref{lma:DT}.
\end{proof}

\begin{proposition}
    Let $\mathbb{D}_1,\ldots,\mathbb{D}_n$ be nef b-divisors over $X$. Then
    \[
    \left(\mathbb{D}_1,\ldots,\mathbb{D}_n \right)\geq \prod_{i=1}^n \left( \vol \mathbb{D}_i\right)^{1/n}.
    \]
\end{proposition}
\begin{proof}
    We may assume that $\vol \mathbb{D}_i>0$ for each $i=1,\ldots,n$ since there is nothing to prove otherwise. In this case, our assertion follows from \cref{prop:logvol}.
\end{proof}

\begin{proposition}\label{prop:DFusc}
    The product in \cref{def:DF_ext} is upper semicontinuous in the following sense. Suppose that $(\mathbb{D}_i^j)_{j\in J}$ are nets of nef b-divisors over $X$ with limits $\mathbb{D}_i$ for each $i=1,\ldots,n$. Then
    \[
    \varlimsup_{j\in J}\left(\mathbb{D}_1^j,\ldots,\mathbb{D}_n^j \right)\leq  \left( \mathbb{D}_1,\ldots,\mathbb{D}_n \right).
    \]
\end{proposition}
\begin{proof}
    \textbf{Step~1}. We first assume that the $\mathbb{D}_i^j$'s and the $\mathbb{D}_i$'s are all big. 

    Take $\mathcal{I}$-good non-divisorial closed positive $(1,1)$-currents $T_i^j$ and $T_i$ so that $\mathbb{D}(T_i^j)=\mathbb{D}_i^j$ and $\mathbb{D}(T_i)=\mathbb{D}_i$. Note that by our assumption and the proof of \cref{thm:bdiv_current_corr}, for any prime divisor $E$ over $X$, we have
    \[
    \lim_{j\in J} \nu(T_i^j,E)=\nu(T_i,E).
    \]
    So our assertion follows from \cref{thm:usc_mix_vol}.

    \textbf{Step~2}. Next we handle the general case. 
    
    Take a Kähler form $\omega$ on $X$. Then by \cref{lma:perturbDF}, for any $\epsilon\in (0,1]$, we have
    \[
    \begin{aligned}
          \varlimsup_{j\in J}\left(\mathbb{D}_1^j,\ldots,\mathbb{D}_n^j \right)
        \leq & \varlimsup_{j\in J}\left(\mathbb{D}_1^j+\epsilon \mathbb{D}(\omega),\ldots,\mathbb{D}_n^j+\epsilon \mathbb{D}(\omega) \right)\\
        \leq & \left( \mathbb{D}_1+\epsilon \mathbb{D}(\omega),\ldots,\mathbb{D}_n+\epsilon \mathbb{D}(\omega) \right)\\
        \leq & \left( \mathbb{D}_1,\ldots,\mathbb{D}_n \right)+C\epsilon.
    \end{aligned}
    \]
    Since $\epsilon$ is arbitrary, our assertion follows.
\end{proof}

\begin{proposition}\label{prop:DFdecnet}
    The product in \cref{def:DF_ext} is continuous along decreasing nets in each variable. In other words, if $(\mathbb{D}_i^j)_{j\in J}$ ($i=1,\ldots,n$) are decreasing nets of nef b-divisors over $X$ with limits $\mathbb{D}_i$. Then
    \[
    \lim_{j\in J} \left(\mathbb{D}_1^j,\ldots,\mathbb{D}_n^j\right)=\left(\mathbb{D}_1,\ldots,\mathbb{D}_n\right).
    \]
\end{proposition}
\begin{proof}
This is a straightforward consequence of \cref{prop:DFmono} and \cref{prop:DFusc}.
\end{proof}

\begin{remark}
As shown in \cite{XiaPPT, Xia24}, this intersection theory coincides with the Dang--Favre theory if $X$ is projective and $\mathbb{D}_1,\ldots,\mathbb{D}_n$ are algebraic.

To be more precise, these papers handled the case where the cohomology classes $\mathbb{D}_{1,X},\ldots,\mathbb{D}_{n,X}$ lie in the Néron--Severi group $\NS^1(X)$. By scaling, the same holds if they lie in the $\mathbb{Q}$-span of $\NS^1(X)$. Finally, by \cref{prop:DFdecnet}, the same holds in general.
\end{remark}

\section{Smooth pull-backs of b-divisors}\label{sec:smpull}
Let $X$ be a connected compact Kähler manifold of dimension $n$. Consider a smooth morphism $f\colon Y\rightarrow X$ of relative dimension $m$ from another connected compact Kähler manifold $Y$. Given a nef b-divisor $\mathbb{D}$ over $X$, we shall define a functorial pull-back $f^*\mathbb{D}$ over $Y$. 

This section is purely of auxiliary purpose. Hence we do not pursue the most general statements. In fact, it is possible to define a pull-back even when $f$ is not smooth, using the non-Archimedean theory of \cite{BJ22}.
In the next section, we will need a special case of the construction in this section, where $Y$ is a projective bundle on $X$.

We first assume that $\mathbb{D}$ is big and nef. Thanks to \cref{thm:bdiv_current_corr}, we can find a non-divisorial closed positive $(1,1)$-current $T$ in $\mathbb{D}_X$ such that $\mathbb{D}(T)=\mathbb{D}$. Moreover, $T$ is unique up to $\mathcal{I}$-equivalence. 

We can therefore define
\[
f^*\mathbb{D}\coloneqq \mathbb{D}(f^*T).
\]
Note that thanks to \cite[Proposition~1.4.5]{Xiabook}, the $\mathcal{I}$-equivalence class of $f^*T$ is independent of the choices of $T$. Hence $f^*\mathbb{D}$ is a well-defined nef b-divisor over $Y$, independent of the choice of $T$.

Observe that $f^*T$ is non-divisorial since this is the case if $f$ is either a projection or étale. In particular, $(f^*\mathbb{D})_Y=f^*\mathbb{D}_X$.

In general, if $\mathbb{D}$ is not necessarily nef, we take a Kähler form $\omega$ on $X$ and define
\[
f^*\mathbb{D}\coloneqq \lim_{\epsilon\to 0+}  f^*\left(\mathbb{D}+\epsilon\mathbb{D}(\omega) \right).
\]
Note that $f^*\left(\mathbb{D}+\epsilon\mathbb{D}(\omega) \right)$ is increasing with respect to $\epsilon>0$, so the limit makes sense. It is clear that this definition is independent of the choice of $\omega$. Observe that 
\begin{equation}\label{eq:fstarroot}
(f^*\mathbb{D})_Y=f^*\mathbb{D}_X.
\end{equation}

\begin{proposition}\label{prop:DfstarT}
    The pull-back $f^*$ defined above is $\mathbb{R}_{\geq 0}$-linear. Moreover, for any closed positive $(1,1)$-current $T$ on $X$, we have
    \begin{equation}\label{eq:DfstarT}
        \mathbb{D}(f^*T)=f^*\mathbb{D}(T).
    \end{equation}
\end{proposition}
\begin{proof}
    The $\mathbb{R}_{\geq 0}$-linearity is obvious. 
    
    We prove \eqref{eq:DfstarT}. 
    Fix a Kähler form $\omega$ on $X$.
    It suffices to handle two cases separately: $T$ is either non-divisorial or divisorial. In the first case, by definition, 
    \[
    f^*\mathbb{D}(T)=\lim_{\epsilon\to 0+} \mathbb{D}\left(f^*(T+\epsilon \omega)\right)=\mathbb{D}(f^*T).
    \]
    Next we assume that $T$ is divisorial, say $T=\sum_i c_i [E_i]$.
    In this case, by \eqref{eq:finitediv_convinf} and \cref{prop:dscontbdiv}, we may assume that $T$ has finitely many components. By linearity, we reduce to the case where $T=[E]$ for some prime divisor $E$ on $X$. In this case, we have $f^*T=[f^{-1}E]$. Hence both sides of \eqref{eq:DfstarT} vanish.
\end{proof}

The pull-back is functorial as expected.
\begin{proposition}
    Let $g\colon Z\rightarrow Y$ be another smooth morphism from a connected compact Kähler manifold $Z$. Then for any nef b-divisor $\mathbb{D}$ over $X$, we have
    \begin{equation}\label{eq:pull_back_func}
        (f\circ g)^*\mathbb{D}=g^*f^*\mathbb{D}.
    \end{equation}
\end{proposition}
\begin{proof}
    We may assume that $\mathbb{D}$ is big. Then there is a non-divisorial closed positive $(1,1)$-current $T\in \mathbb{D}_X$ so that $\mathbb{D}(T)=\mathbb{D}_X$. 
    
    Thanks to \cref{prop:DfstarT}, both sides of \eqref{eq:pull_back_func} are equal to $\mathbb{D}(g^*f^*T)$.
\end{proof}

\begin{proposition}\label{prop:smoothpull}
    Let $\pi\colon X'\rightarrow X$ be a modification. Consider the Cartesian diagram,
    \begin{equation}\label{eq:Cart_modif}
    \begin{tikzcd}
Y' \arrow[d, "\pi_Y"'] \arrow[r, "f'"] \arrow[rd, "\square", phantom] & X' \arrow[d, "\pi"] \\
Y \arrow[r, "f"]                                                      & X.                  
\end{tikzcd}
    \end{equation}
    Then for any nef b-divisor $\mathbb{D}$ over $X$, we have
    \begin{equation}\label{eq:pullbacksmD}
    (f^*\mathbb{D})_{Y'}=f'^*\mathbb{D}_{X'}.
    \end{equation}
\end{proposition}
Thanks to the smoothness of $f$, $\pi_Y$ is also a modification, so the left-hand side of \eqref{eq:pullbacksmD} makes sense.
\begin{proof}
    We may assume that $\mathbb{D}$ is big. Take a non-divisorial closed positive $(1,1)$-current $T$ in $\mathbb{D}_X$ so that $\mathbb{D}=\mathbb{D}(T)$. 
     Since $f'^*$ preserves non-divisorial currents and divisorial currents, we have
    \[
    f'^*\Reg \pi^*T=\Reg \pi_Y^*(f^*T).
    \]
    Therefore,
    \[
     \{\Reg \pi_Y^*(f^*T)\}=f'^*\{\Reg \pi^*T\}.
    \]
    Our assertion follows.
\end{proof}

\begin{proposition}\label{prop:order_pull}
    Let $\mathbb{D},\mathbb{D}'$ be nef b-divisors over $X$ with $\mathbb{D}_X=\mathbb{D}'_X$. 
    Then the following are equivalent:
    \begin{enumerate}
        \item $\mathbb{D}\leq \mathbb{D}'$;
        \item $f^*\mathbb{D}\leq f^*\mathbb{D}'$.
    \end{enumerate}
\end{proposition}
\begin{proof}
     We may assume that $\mathbb{D}$ and $\mathbb{D}'$ are both big. Take non-divisorial closed positive $(1,1)$-currents $T$ and $T'$ in $\mathbb{D}_X$ such that $\mathbb{D}=\mathbb{D}(T)$ and $\mathbb{D}'=\mathbb{D}(T')$. 

(1) $\implies$ (2). Assume (1).
   It follows from \cref{cor:Dorder_Iorder} that $T\preceq_{\mathcal{I}} T'$. By \cite[Proposition~1.4.5]{Xiabook}, we have $f^*T\preceq_{\mathcal{I}} f^*T'$, hence by \eqref{eq:fstarroot} and \cref{cor:Dorder_Iorder} again, we find $f^*\mathbb{D}\leq f^*\mathbb{D}'$.

(2) $\implies$ (1). Assume (2). Fix a prime divisor $E$ over $X$.
It suffices to show that 
\begin{equation}\label{eq:nu_ineq}
\nu(T,E)\geq \nu(T',E). 
\end{equation}
Take a modification $\pi\colon X'\rightarrow X$ so that $E$ is a prime divisor on $X'$. Form the Cartesian diagram \eqref{eq:Cart_modif}. Then by \cref{cor:Dorder_Iorder},
\[
\nu(f^*T,f'^{-1}E)\geq \nu(f^*T',f'^{-1}E),
\]
which, thanks to \cite[Proposition~1.4.5]{Xiabook}, is just \eqref{eq:nu_ineq}.
\end{proof}

\begin{proposition}\label{prop:dscont_smoothpull}
    Let $\theta$ be a smooth closed real $(1,1)$-form on $X$ representing a big cohomology class. Let $(\varphi_i)_{i\in I}$ be a net in $\PSH(X,\theta)$ and $\varphi\in \PSH(X,\theta)$. Assume that $\varphi_i\xrightarrow{d_S}\varphi$, then 
    \[
    f^*\varphi_i\xrightarrow{d_S} f^*\varphi.
    \]
\end{proposition}
\begin{proof}
    Since $\PSH(X,\theta)$ is a pseudometric space, we may assume that $(\varphi_i)_i$ is a sequence. Replacing $\theta$ by $\theta+\omega$ for some Kähler form $\omega$ on $X$, we may assume that the non-pluripolar masses of the $\varphi_i$'s are bounded from below by a positive constant. Then it follows from \cite[Proposition~4.2]{DDNL18mono} and \cite[Corollary~6.2.11]{Xiabook} that we may assume without loss of generality that $(\varphi_i)_i$ is either increasing or decreasing. 
    
    The increasing case follows from \cite[Corollary~6.2.3]{Xiabook}.
    We assume that $(\varphi_i)_i$ is a decreasing sequence. Fix a Kähler form $\Omega$ on $Y$. By \cite[Corollary~6.2.5]{Xiabook}, it remains to argue that
    \[
    \lim_{i\to\infty}\int_Y \left( f^*\theta+\Omega+\ddc f^*\varphi_i \right)^{n+m}=\int_Y \left( f^*\theta+\Omega+\ddc f^*\varphi \right)^{n+m}.
    \]
    After a binomial expansion, it suffices to show that for any $a=0,\ldots,n$, we have
    \[
    \lim_{i\to\infty}\int_Y \left( f^*\theta+\ddc f^*\varphi_i \right)^{a}\wedge \Omega^{n+m-a}=\int_Y \left( f^*\theta+\ddc f^*\varphi \right)^{a}\wedge \Omega^{n+m-a},
    \]
    or equivalently,
    \[
    \lim_{i\to\infty}\int_X \left( \theta+\ddc \varphi_i \right)^{a}\wedge f_*\Omega^{n+m-a}=\int_X \left( \theta+\ddc \varphi \right)^{a}\wedge f_*\Omega^{n+m-a}.
    \]
    Since $f$ is smooth, the form $f_*\Omega^{n+m-a}$ is smooth as well. Our assertion then follows from \cite[Theorem~1.9]{Xia24}.
\end{proof}

\begin{corollary}\label{cor:cont_smoothpull}
    Let $(T_i)_{i\in I}$ be a net of closed positive $(1,1)$-currents on $X$ and $T$ be a closed positive $(1,1)$-current on $X$. Assume that $T_i\implies T$, then $f^*T_i\implies f^*T$.
\end{corollary}
\begin{proof}
    This is an immediate consequence of \cref{prop:dscont_smoothpull}.
\end{proof}

\section{The trace operator of b-divisors}\label{sec:resbdiv}
Let $X$ be a connected compact Kähler manifold of dimension $n$ and $Z$ be a smooth irreducible analytic set of dimension $m$ in $X$. Let $\mathbb{D}$ be a nef b-divisor over $X$. 

We will study the problem of restricting nef b-divisors over $X$ to $Z$ in this section. This problem has been studied in the analytic setting in \cite{DX24}. We shall follow the slightly different approach as studied in \cite[Chapter~8]{Xiabook}, which is better behaved in the zero mass case.

\subsection{The analytic theory}
Let $T$ be a closed positive $(1,1)$-current on $X$ representing a cohomology class $\alpha$. Assume that $\nu(T,Z)=0$.

Consider a quasi-equisingular approximation $(T_j)_j$ of $T$, where the currents $T_j$ are not necessarily in $\alpha$. Then $\nu(T_j,Z)=0$ and hence $T_j|_Z$ makes sense. We then define $\Tr_Z T$ as any closed positive $(1,1)$-current on $Z$ such that $T_j|_Z\xrightarrow{d_S} \Tr_Z T$.  

One can show that $\Tr_Z T$ is always well-defined modulo $P$-equivalence and is independent of the choice of the sequence $(T_j)_j$. 

If furthermore, $T$ is a Kähler current, then $\Tr_Z T$ can be represented by Kähler current in $\alpha|_Z$.

The details can be found in \cite[Chapter~8]{Xiabook}.

\subsection{The codimension \texorpdfstring{$1$}{1} case}

We assume that $Z$ is a divisor so that $m=n-1$. 

For the moment, let us assume that $\mathbb{D}$ is a Cartier nef b-divisor. Let $(\pi\colon Y\rightarrow X,\alpha)$ be a realization of $\mathbb{D}$. 

Let $Z_Y$ denote the strict transform of $Z$ and $p_Y\colon Z_Y\rightarrow Z$ denotes the restriction of $\pi$. The notations are summarized in the commutative diagram:
\begin{equation}\label{eq:realization}
\begin{tikzcd}
Z_Y \arrow[r, hook] \arrow[d, "p_Y"'] & Y \arrow[d, "\pi"] \\
Z \arrow[r, hook]                     & X.                 
\end{tikzcd}
\end{equation}

After replacing $\pi$ by a further modification, we may assume that $Z_Y$ is smooth. This follows from the embedded resolution \cite{BM97, Wlo09}.
In this case, we define the \emph{trace} $\Tr_Z \mathbb{D}$ of $\mathbb{D}$ on $Z$ as the nef Cartier b-divisor over $Z$ realized by $(p_Y,\alpha|_{Z_Y})$. 
Note that we are slightly abusing our language since $p_Y$ is not a modification in general. To be more precise, here we mean that for any modification $Z'\rightarrow Z$ dominating $Z_Y$, $\Tr_Z \mathbb{D}$ is defined as the nef Cartier b-divisor over $Z$ realized by $(Z'\rightarrow Z,\beta)$, where $\beta$ is the pull-back of $\alpha|_{Z_Y}$.

\begin{lemma}
    Assume that $\mathbb{D}$ is a Cartier nef b-divisor, then $\Tr_Z \mathbb{D}$ defined above is independent of the choice of $\pi$.
\end{lemma}
\begin{proof}
    Given a different realization $(\pi'\colon Y'\rightarrow X,\alpha')$ of $\mathbb{D}$, we want to show that it defines the same $\Tr_Z\mathbb{D}$. We may assume that $\pi'$ dominates $\pi$ so that we have a commutative diagram:
    \[
    \begin{tikzcd}
Z_{Y'} \arrow[r, hook] \arrow[d, "\tau"] \arrow[dd, "p_{Y'}"', bend right] & Y' \arrow[dd, "\pi'", bend left] \arrow[d, "\sigma"] \\
Z_Y \arrow[r, hook] \arrow[d, "p_Y"]                                       & Y \arrow[d, "\pi"']                                  \\
Z \arrow[r, hook]                                                          & X.                                                  
\end{tikzcd}
    \]
    The notations $\tau,\sigma,p_{Y'}$ have the obvious meanings. We may assume that $Z_Y$ and $Z_{Y'}$ are both smooth. 

    Our assertion becomes the following:
    \[
    \tau^*\left( \alpha|_{Z_Y} \right)= \left( \sigma^*\alpha \right)|_{Z_{Y'}},
    \]
    which is obvious since the upper square in the diagram commutes.
\end{proof}

\begin{proposition}\label{prop:trace_mono}
    Let $\mathbb{D}'$ be another nef Cartier b-divisor over $X$. Assume that $\mathbb{D'}\leq \mathbb{D}$ and $\mathbb{D}'_X=\mathbb{D}_X$, then $\Tr_Z\mathbb{D}'\leq \Tr_Z \mathbb{D}$.
\end{proposition}
\begin{proof}
    We take realizations $(\pi\colon Y\rightarrow X,\alpha')$ and $(\pi,\alpha)$ of $\mathbb{D}'$ and $\mathbb{D}$ on the same modification. Then by assumption $\alpha\geq \alpha'$, and $\alpha-\alpha'$ is represented by an effective $\mathbb{R}$-divisor not containing $Z_Y$ in its support. It follows that $\alpha|_{Z_Y}\geq \alpha'|_{Z_Y}$. Therefore, our assertion follows. 
\end{proof}

\begin{lemma}\label{lma:TrZDTi1}
    Let $T$ be a closed positive $(1,1)$-current with analytic singularities
    Then 
    \begin{equation}\label{eq:TrZDTi1}
    \Tr_Z \mathbb{D}(T)=\mathbb{D}\left(\Tr_Z \left(T-\nu(T,Z)[Z]\right) \right).
    \end{equation}
\end{lemma}
\begin{proof}
    Let $\pi\colon Y\rightarrow X$ be a modification so that 
\[
\pi^*T=[D]+R,
\]
where $D$ is an effective $\mathbb{Q}$-divisor and $R$ is a closed positive $(1,1)$-current with locally bounded potential. 
We may assume that the strict transform $Z_Y$ of $Z$ is smooth. Then by definition, both sides of \eqref{eq:TrZDTi1} are Cartier nef b-divisors realized by $(Z_Y,\{R\}|_{Z_Y})$.
\end{proof}

In general, when $\mathbb{D}$ is nef and big but not necessarily Cartier, take an $\mathcal{I}$-good non-divisorial Kähler current $T\in \mathbb{D}_X$ so that $\mathbb{D}=\mathbb{D}(T)$. Consider a quasi-equisingular approximation $(T_i)_i$ of $T$ in the same cohomology class as $T$, we define
\[
\Tr_Z \mathbb{D}=\lim_{i\to\infty} \Tr_Z \mathbb{D}(T_i).
\]
Thanks to \cref{prop:trace_mono}, the right-hand side is a decreasing sequence and hence the limit exists.
\begin{lemma}
    The definition of $\Tr_Z \mathbb{D}$ is independent of the choices we made, and
    \begin{equation}\label{eq:trZD}
    \Tr_Z \mathbb{D}=\mathbb{D}(\Tr_Z T).
    \end{equation}
\end{lemma}
\begin{proof}
    It suffices to prove \eqref{eq:trZD}, but this follows from \cref{prop:dscontbdiv} and the $d_S$-continuity of the trace operator along decreasing sequences.
\end{proof}

More generally, if $\mathbb{D}$ is just nef, we take a Kähler form $\omega$ on $X$, and let
\[
\Tr_Z \mathbb{D}\coloneqq \lim_{\epsilon\to 0+}\Tr_Z \left( \mathbb{D}+\epsilon \mathbb{D}(\omega) \right). 
\]
This definition is independent of the choice of $\omega$.

\begin{theorem}\label{thm:trac_codim1}
    Let $T$ be a closed positive $(1,1)$-current such that $\Tr_Z \left(T-\nu(T,Z)[Z]\right)$ can be represented by a closed positive $(1,1)$-current in $\{T-\nu(T,Z)[Z]\}|_Z$. Take such a representative.
    Then 
    \[
    \Tr_Z \mathbb{D}(T)=\mathbb{D}\left(\Tr_Z \left(T-\nu(T,Z)[Z]\right) \right).
    \]
\end{theorem}
\begin{proof}
Replacing $T$ by $T+\epsilon\omega$ for some $\epsilon>0$ and some Kähler form $\omega$, we may assume that $\mathbb{D}(T)$ is big.

Replacing $T$ by $T-\nu(T,Z)[Z]$, we may assume that $\nu(T,Z)=0$. Then we need to show that
\[
\Tr_Z \mathbb{D}(T)=\mathbb{D}\left( \Tr_Z T \right).
\]
Here $\Tr_Z T$ is in $\{T\}|_Z$.

Let $(T_j)_j$ be a quasi-equisingular approximation of $T$ in the same cohomology class as $T$. Then $T_j\xrightarrow{d_S} T$. Hence by \cref{prop:dscontbdiv}, we have
\[
\mathbb{D}(T_i)\to \mathbb{D}(T).
\]
By definition and \cref{prop:dscontbdiv},
\[
\Tr_Z \mathbb{D}(T)=\lim_{i\to\infty}\Tr_Z \mathbb{D}(T_i),\quad \mathbb{D}\left( \Tr_Z T \right)=\lim_{i\to\infty} \mathbb{D}(T_i|_Z).
\]
Hence our assertion follows from \cref{lma:TrZDTi1}.

\end{proof}

\subsection{The higher codimension case}
Now assume that $Z$ has codimension at least $2$.

In this case, similar to the analytic theory, we cannot restrict a general nef b-divisor. 

We consider the following commutative diagram:
\begin{equation}\label{eq:BlZX}
\begin{tikzcd}
E \arrow[r, hook] \arrow[d, "q" '] & \Bl_Z X \arrow[d, "p"] \\
Z \arrow[r, hook]          & X   ,                 
\end{tikzcd}
\end{equation}
where $p\colon \Bl_Z X\rightarrow X$ is the blow-up of $X$ along $Z$ and $E$ is the exceptional divisor. Note that $q\colon E\rightarrow Z$ can be naturally identified with the projectivized normal bundle of $Z$ in $X$.

Let $\mathbb{D}$ be a nef b-divisor over $X$ such that
\begin{equation}\label{eq:D_Bl}
    \mathbb{D}_{\Bl_Z X}=p^*\mathbb{D}_X.
\end{equation}
Assume \eqref{eq:D_Bl}, then the trace can be defined. To do so, we shall rely on the analytic theory.

\begin{proposition}
    Let $\mathbb{D}$ be a nef b-divisor over $X$ satisfying \eqref{eq:D_Bl}. Then there is a unique nef b-divisor $\Tr_Z \mathbb{D}$ over $Z$ such that
    \begin{equation}\label{eq:Tr_def}
    q^*\Tr_Z \mathbb{D}=\Tr_E \mathbb{D},
    \end{equation}
    where $\mathbb{D}$ is regarded as a nef b-divisor over $\Bl_Z X$ in the obvious way.
\end{proposition}
The pull-back $q^*$ is defined in \cref{sec:smpull}.

We first recall the following decomposition:
\begin{equation}\label{eq:H11_projb}
    \mathrm{H}^{1,1}(E,\mathbb{R})=\mathrm{H}^{1,1}(Z,\mathbb{R})\oplus \mathbb{R}\zeta,
\end{equation}
where $\zeta$ is the tautological class of the projective bundle $q$. See \cite[Proposition~3.3]{RYY19} for example. This decomposition also explains why we need to impose the condition \eqref{eq:D_Bl}.
\begin{proof}
    Thanks to \cref{prop:smoothpull}, \eqref{eq:D_Bl} and \eqref{eq:H11_projb}, the root of $\Tr_Z \mathbb{D}$ is necessarily the first component of $(\Tr_E\mathbb{D})_E$ with respect to the decomposition \eqref{eq:H11_projb}.
    By \cref{prop:order_pull}, the nef b-divisor $\Tr_Z \mathbb{D}$ is unique if it exists.

    Fix a Kähler form $\omega$ on $X$. It suffices to prove the existence of $\Tr_Z (\mathbb{D}+\epsilon \mathbb{D}(\omega))$ for any $\epsilon>0$. In fact, if we have established these existence, then thanks to \cref{prop:order_pull}, we know that $\Tr_Z (\mathbb{D}+\epsilon \mathbb{D}(\omega))$ is increasing with respect to $\epsilon$, hence defining 
    \[
    \Tr_Z \mathbb{D}\coloneqq \lim_{\epsilon \to 0+}\Tr_Z \left(\mathbb{D}+\epsilon \mathbb{D}(\omega)\right)
    \]
    would suffice.

    Therefore, we may assume that there is a non-divisorial Kähler current $T$ in $\mathbb{D}_X$ such that $\mathbb{D}=\mathbb{D}(T)$. Then \eqref{eq:D_Bl} translates into $\nu(T,Z)=0$. In particular, $\Tr_Z T$ is defined and can be represented by a Kähler current in $\{T\}|_Z$. We fix such a representative.
    We claim that in fact
    \begin{equation}\label{eq:TrED}
    \Tr_E\mathbb{D}=q^*\mathbb{D}(\Tr_Z T).
    \end{equation}
    In fact, due to \cref{prop:DfstarT}, we know that
    \[
    q^*\mathbb{D}(\Tr_Z T)=\mathbb{D}(q^*\Tr_Z T).
    \]
    Thanks to \cref{thm:trac_codim1}, \eqref{eq:TrED} translates into
    \begin{equation}\label{eq:TrED2}
    \mathbb{D}(q^*\Tr_Z T)=\mathbb{D}(\Tr_E (p^*T)).
    \end{equation}
    Now \cref{cor:cont_smoothpull} and \cref{prop:dscontbdiv} allow us to reduce to the case where $T$ has analytic singularities, and \eqref{eq:TrED2} finally reduces to
    \[
    \mathbb{D}\left(q^*(T|_Z)\right)=\mathbb{D}\left((p^*T)|_E\right),
    \]
    which follows immediately from the commutativity of \eqref{eq:BlZX}.
\end{proof}

\begin{definition}
    Let $\mathbb{D}$ be a nef b-divisor over $X$ satisfying \eqref{eq:D_Bl}. Then $\Tr_Z\mathbb{D}$ is defined as the unique nef b-divisor over $Z$ such that \eqref{eq:Tr_def} holds.
\end{definition}
One can easily deduce the basic properties of the trace $\Tr_Z\mathbb{D}$ from the analytic theory of trace operators. We omit these transparent translations.

The trace operator of b-divisors has a natural explanation in terms of non-Archimedean metrics, see \cite{Xia23Operations}.

\printbibliography

@misc{Xiabdiv2,
    title={Transcendental b-divisors {II}},
    author={Xia, Mingchen},
    year={2025},
    howpublished={\href{https://mingchenxia.github.io/Papers/TB2.pdf}{link}}
}

@incollection {Xia23Operations,
    AUTHOR = {Xia, Mingchen},
     TITLE = {Operations on transcendental non-{A}rchimedean metrics},
 BOOKTITLE = {Convex and complex: perspectives on positivity in geometry},
    SERIES = {Contemp. Math.},
    VOLUME = {810},
     PAGES = {271--293},
 PUBLISHER = {Amer. Math. Soc., Providence, RI},
      YEAR = {2025},
      ISBN = {978-1-4704-7338-9},
   MRCLASS = {32Q99 (32P05)},
  MRNUMBER = {4853201},
       DOI = {10.1090/conm/810/16216},
       URL = {https://doi.org/10.1090/conm/810/16216},
}

@article {Tru24,
    AUTHOR = {Trusiani, Antonio},
     TITLE = {A relative {Y}au-{T}ian-{D}onaldson conjecture and stability
              thresholds},
   JOURNAL = {Adv. Math.},
  FJOURNAL = {Advances in Mathematics},
    VOLUME = {441},
      YEAR = {2024},
     PAGES = {Paper No. 109537, 95},
      ISSN = {0001-8708,1090-2082},
   MRCLASS = {32Q26 (14J45 32Q20 32U05)},
  MRNUMBER = {4708148},
       DOI = {10.1016/j.aim.2024.109537},
       URL = {https://doi.org/10.1016/j.aim.2024.109537},
}

@article {DX24,
    TITLE = {The trace operator of quasi-plurisubharmonic functions on compact {K}\"ahler manifolds},
    AUTHOR = {Darvas, Tam\'as and Xia, Mingchen},
    year={2024},
    eprint={2403.08259},
    archivePrefix={arXiv},
    primaryClass={math.DG},
    journal={Trans. Amer. Math. Soc. (to appear)}
}

@article {RYY19,
    AUTHOR = {Rao, Sheng and Yang, Song and Yang, Xiangdong},
     TITLE = {Dolbeault cohomologies of blowing up complex manifolds},
   JOURNAL = {J. Math. Pures Appl. (9)},
  FJOURNAL = {Journal de Math\'ematiques Pures et Appliqu\'ees. Neuvi\`eme
              S\'erie},
    VOLUME = {130},
      YEAR = {2019},
     PAGES = {68--92},
      ISSN = {0021-7824,1776-3371},
   MRCLASS = {32S45 (14D07 14E05 18G40 32C36)},
  MRNUMBER = {4001628},
MRREVIEWER = {George-Ionu\c t\ Ioni\c t\u a},
       DOI = {10.1016/j.matpur.2019.01.016},
       URL = {https://doi.org/10.1016/j.matpur.2019.01.016},
}

@article {CT22,
    AUTHOR = {Collins, Tristan C. and Tosatti, Valentino},
     TITLE = {Restricted volumes on {K}\"{a}hler manifolds},
   JOURNAL = {Ann. Fac. Sci. Toulouse Math. (6)},
  FJOURNAL = {Annales de la Facult\'{e} des Sciences de Toulouse. Math\'{e}matiques.
              S\'{e}rie 6},
    VOLUME = {31},
      YEAR = {2022},
    NUMBER = {3},
     PAGES = {907--947},
      ISSN = {0240-2963},
   MRCLASS = {32Q15 (14C17 53C55)},
  MRNUMBER = {4452256},
       DOI = {10.5802/afst.170},
       URL = {https://doi.org/10.5802/afst.170},
}

@article {BT87,
    AUTHOR = {Bedford, Eric and Taylor, B. A.},
     TITLE = {Fine topology, \v{S}ilov boundary, and {$(dd^c)^n$}},
   JOURNAL = {J. Funct. Anal.},
  FJOURNAL = {Journal of Functional Analysis},
    VOLUME = {72},
      YEAR = {1987},
    NUMBER = {2},
     PAGES = {225--251},
      ISSN = {0022-1236},
   MRCLASS = {32F05 (31C15)},
  MRNUMBER = {886812},
MRREVIEWER = {J. Siciak},
       DOI = {10.1016/0022-1236(87)90087-5},
       URL = {https://doi.org/10.1016/0022-1236(87)90087-5},
}

@article {DRWNXZ,
    AUTHOR = {Darvas, Tam\'as and Reboulet, R\'emi and Witt Nystr\"om, David
              and Xia, Mingchen and Zhang, Kewei},
     TITLE = {Transcendental {O}kounkov bodies},
   JOURNAL = {J. Differential Geom.},
  FJOURNAL = {Journal of Differential Geometry},
    VOLUME = {132},
      YEAR = {2026},
    NUMBER = {1},
     PAGES = {135--178},
      ISSN = {0022-040X,1945-743X},
   MRCLASS = {14C20 (32C99 32Q15 53C55)},
  MRNUMBER = {5007919},
       DOI = {10.4310/jdg/1766433802},
       URL = {https://doi.org/10.4310/jdg/1766433802},
       label={DRWN+}
}

@misc{McC21,
  title={Plurisupported Currents on Compact Kähler Manifolds},
  author={McCleerey, Nicholas},
  archivePrefix = {arXiv},
  eprint={2106.12017},
  year={2021},
  primaryClass={math.CV}
}

@article {Xia24,
    AUTHOR = {Xia, Mingchen},
     TITLE = {Non-pluripolar products on vector bundles and {C}hern-{W}eil
              formulae},
   JOURNAL = {Math. Ann.},
  FJOURNAL = {Mathematische Annalen},
    VOLUME = {390},
      YEAR = {2024},
    NUMBER = {3},
     PAGES = {3239--3316},
      ISSN = {0025-5831,1432-1807},
   MRCLASS = {32J25 (14C17 14G40 32U40)},
  MRNUMBER = {4803453},
       DOI = {10.1007/s00208-024-02838-4},
       URL = {https://doi.org/10.1007/s00208-024-02838-4},
}

@article {DX21,
    AUTHOR = {Darvas, Tam\'as and Xia, Mingchen},
     TITLE = {The volume of pseudoeffective line bundles and partial
              equilibrium},
   JOURNAL = {Geom. Topol.},
  FJOURNAL = {Geometry \& Topology},
    VOLUME = {28},
      YEAR = {2024},
    NUMBER = {4},
     PAGES = {1957--1993},
      ISSN = {1465-3060,1364-0380},
   MRCLASS = {32L05 (32W20 53C55)},
  MRNUMBER = {4777706},
MRREVIEWER = {S\l awomir\ Dinew},
       DOI = {10.2140/gt.2024.28.1957},
       URL = {https://doi.org/10.2140/gt.2024.28.1957},
}

@article {DF20,
    AUTHOR = {Dang, Nguyen-Bac and Favre, Charles},
     TITLE = {Intersection theory of nef {$b$}-divisor classes},
   JOURNAL = {Compos. Math.},
  FJOURNAL = {Compositio Mathematica},
    VOLUME = {158},
      YEAR = {2022},
    NUMBER = {7},
     PAGES = {1563--1594},
      ISSN = {0010-437X,1570-5846},
   MRCLASS = {14C20 (14E05)},
  MRNUMBER = {4476829},
MRREVIEWER = {Piotr\ Pokora},
       DOI = {10.1112/s0010437x22007515},
       URL = {https://doi.org/10.1112/s0010437x22007515},
}

@article {Xia21,
    AUTHOR = {Xia, Mingchen},
     TITLE = {Partial {O}kounkov bodies and {D}uistermaat-{H}eckman measures
              of non-{A}rchimedean metrics},
   JOURNAL = {Geom. Topol.},
  FJOURNAL = {Geometry \& Topology},
    VOLUME = {29},
      YEAR = {2025},
    NUMBER = {3},
     PAGES = {1283--1344},
      ISSN = {1465-3060,1364-0380},
   MRCLASS = {14M25 (32L05 32P05)},
  MRNUMBER = {4918108},
       DOI = {10.2140/gt.2025.29.1283},
       URL = {https://doi.org/10.2140/gt.2025.29.1283},
}

@article {DF20a,
    AUTHOR = {Dang, Nguyen-Bac and Favre, Charles},
     TITLE = {Spectral interpretations of dynamical degrees and
              applications},
   JOURNAL = {Ann. of Math. (2)},
  FJOURNAL = {Annals of Mathematics. Second Series},
    VOLUME = {194},
      YEAR = {2021},
    NUMBER = {1},
     PAGES = {299--359},
      ISSN = {0003-486X,1939-8980},
   MRCLASS = {37F80 (14E05 32H50)},
  MRNUMBER = {4276288},
MRREVIEWER = {Mattias\ Jonsson},
       DOI = {10.4007/annals.2021.194.1.5},
       URL = {https://doi.org/10.4007/annals.2021.194.1.5},
}

@article {DXZ23,
    AUTHOR = {Darvas, Tam\'{a}s and Xia, Mingchen and Zhang, Kewei},
     TITLE = {A transcendental approach to non-{A}rchimedean metrics of
              pseudoeffective classes},
   JOURNAL = {Comment. Math. Helv.},
  FJOURNAL = {Commentarii Mathematici Helvetici. A Journal of the Swiss
              Mathematical Society},
    VOLUME = {100},
      YEAR = {2025},
    NUMBER = {2},
     PAGES = {269--322},
      ISSN = {0010-2571},
   MRCLASS = {32U15 (30G06 32P05 32Q20)},
  MRNUMBER = {4888079},
       DOI = {10.4171/cmh/586},
       URL = {https://doi.org/10.4171/cmh/586},
}

@article {BJ22,
    AUTHOR = {Boucksom, S\'{e}bastien and Jonsson, Mattias},
     TITLE = {Global pluripotential theory over a trivially valued field},
   JOURNAL = {Ann. Fac. Sci. Toulouse Math. (6)},
  FJOURNAL = {Annales de la Facult\'{e} des Sciences de Toulouse.
              Math\'{e}matiques. S\'{e}rie 6},
    VOLUME = {31},
      YEAR = {2022},
    NUMBER = {3},
     PAGES = {647--836},
      ISSN = {0240-2963,2258-7519},
   MRCLASS = {32P05 (14G22 14T20 31C05 32U05)},
  MRNUMBER = {4452253},
MRREVIEWER = {Jackson\ S.\ Morrow},
       DOI = {10.5802/afst.170},
       URL = {https://doi.org/10.5802/afst.170},
}

@article {BBGHdJ21,
    AUTHOR = {Botero, Ana Mar\'ia and Burgos Gil, Jos\'e{} Ignacio and
              Holmes, David and de Jong, Robin},
     TITLE = {Chern-{W}eil and {H}ilbert-{S}amuel formulae for singular
              {H}ermitian line bundles},
   JOURNAL = {Doc. Math.},
  FJOURNAL = {Documenta Mathematica},
    VOLUME = {27},
      YEAR = {2022},
     PAGES = {2563--2624},
      ISSN = {1431-0635,1431-0643},
   MRCLASS = {14C17 (32U05 32U25)},
  MRNUMBER = {4574244},
MRREVIEWER = {Martin\ L.\ Sera},
       DOI = {10.4171/dm/x36},
       URL = {https://doi.org/10.4171/dm/x36},
label={BBGHdJ}
}

@misc {Xiabook,
    TITLE = {Singularities in global pluripotential theory},
    AUTHOR = {Xia, Mingchen},
    URL = {https://mingchenxia.github.io/Lectures/SGPT2.pdf}
}

@misc{Dembook2,
  title={Complex analytic and differential geometry},
  author={Demailly, J.-P.},
  howpublished={\url{https://www-fourier.ujf-grenoble.fr/~demailly/manuscripts/agbook.pdf}},
  year={2012}
}

@article{XiaPPT,
      title={Pluripotential-theoretic stability thresholds},
      author={Xia, Mingchen},
      year={2022},
      Journal= {IMRN},
      Fjournal={International Mathematics Research Notices},
      issn = {1073-7928},
      doi = {10.1093/imrn/rnac186},
      url = {https://doi.org/10.1093/imrn/rnac186},
      year = {2022},
}

@article {Hir75,
    AUTHOR = {Hironaka, Heisuke},
     TITLE = {Flattening theorem in complex-analytic geometry},
   JOURNAL = {Amer. J. Math.},
  FJOURNAL = {American Journal of Mathematics},
    VOLUME = {97},
      YEAR = {1975},
     PAGES = {503--547},
      ISSN = {0002-9327},
   MRCLASS = {32C45},
  MRNUMBER = {393556},
MRREVIEWER = {Margherita Galbiati},
       DOI = {10.2307/2373721},
       URL = {https://doi.org/10.2307/2373721},
}

@article {BM97,
    AUTHOR = {Bierstone, Edward and Milman, Pierre D.},
     TITLE = {Canonical desingularization in characteristic zero by blowing
              up the maximum strata of a local invariant},
   JOURNAL = {Invent. Math.},
  FJOURNAL = {Inventiones Mathematicae},
    VOLUME = {128},
      YEAR = {1997},
    NUMBER = {2},
     PAGES = {207--302},
      ISSN = {0020-9910},
   MRCLASS = {14E15 (32S45)},
  MRNUMBER = {1440306},
MRREVIEWER = {Joseph Lipman},
       DOI = {10.1007/s002220050141},
       URL = {https://doi.org/10.1007/s002220050141},
}

@article {Dem85,
    AUTHOR = {Demailly, Jean-Pierre},
     TITLE = {Mesures de {M}onge-{A}mp\`ere et caract\'{e}risation g\'{e}om\'{e}trique des
              vari\'{e}t\'{e}s alg\'{e}briques affines},
   JOURNAL = {M\'{e}m. Soc. Math. France (N.S.)},
  FJOURNAL = {M\'{e}moires de la Soci\'{e}t\'{e} Math\'{e}matique de France. Nouvelle S\'{e}rie},
    NUMBER = {19},
      YEAR = {1985},
     PAGES = {124},
      ISSN = {0037-9484},
   MRCLASS = {32H35 (32C10 32F05)},
  MRNUMBER = {813252},
MRREVIEWER = {G. M. Khenkin},
}

@article {DX22,
    AUTHOR = {Darvas, Tam\'{a}s and Xia, Mingchen},
     TITLE = {The closures of test configurations and algebraic singularity
              types},
   JOURNAL = {Adv. Math.},
  FJOURNAL = {Advances in Mathematics},
    VOLUME = {397},
      YEAR = {2022},
     PAGES = {Paper No. 108198, 56},
      ISSN = {0001-8708},
   MRCLASS = {32W20 (32U05 53C55)},
  MRNUMBER = {4366232},
       DOI = {10.1016/j.aim.2022.108198},
       URL = {https://doi.org/10.1016/j.aim.2022.108198},
}

@article {BDPP13,
    AUTHOR = {Boucksom, S\'{e}bastien and Demailly, Jean-Pierre and P\u{a}un, Mihai
              and Peternell, Thomas},
     TITLE = {The pseudo-effective cone of a compact {K}\"{a}hler manifold and
              varieties of negative {K}odaira dimension},
   JOURNAL = {J. Algebraic Geom.},
  FJOURNAL = {Journal of Algebraic Geometry},
    VOLUME = {22},
      YEAR = {2013},
    NUMBER = {2},
     PAGES = {201--248},
      ISSN = {1056-3911},
   MRCLASS = {14E99 (32J18 32L05 53C26)},
  MRNUMBER = {3019449},
MRREVIEWER = {Thomas Eckl},
       DOI = {10.1090/S1056-3911-2012-00574-8},
       URL = {https://doi.org/10.1090/S1056-3911-2012-00574-8},
}

@article {DDNLmetric,
    AUTHOR = {Darvas, Tam\'{a}s and Di Nezza, Eleonora and Lu, Hoang-Chinh},
     TITLE = {The metric geometry of singularity types},
   JOURNAL = {J. Reine Angew. Math.},
  FJOURNAL = {Journal f\"{u}r die Reine und Angewandte Mathematik. [Crelle's
              Journal]},
    VOLUME = {771},
      YEAR = {2021},
     PAGES = {137--170},
      ISSN = {0075-4102},
   MRCLASS = {32U05 (32W20 53C55)},
  MRNUMBER = {4234098},
MRREVIEWER = {S\l awomir Dinew},
       DOI = {10.1515/crelle-2020-0019},
       URL = {https://doi.org/10.1515/crelle-2020-0019},
       label={DDNL}
}

@article {WN19,
    AUTHOR = {Witt Nystr\"{o}m, David},
     TITLE = {Monotonicity of non-pluripolar {M}onge-{A}mp\`ere masses},
   JOURNAL = {Indiana Univ. Math. J.},
  FJOURNAL = {Indiana University Mathematics Journal},
    VOLUME = {68},
      YEAR = {2019},
    NUMBER = {2},
     PAGES = {579--591},
      ISSN = {0022-2518},
   MRCLASS = {32W20 (32Q15 32U05 35J96)},
  MRNUMBER = {3951074},
MRREVIEWER = {Rafa\l  Czy\.{z}},
       DOI = {10.1512/iumj.2019.68.7630},
       URL = {https://doi.org/10.1512/iumj.2019.68.7630},
       label={WN}
}

@article {BBJ21,
    AUTHOR = {Berman, Robert J. and Boucksom, S\'{e}bastien and Jonsson,
              Mattias},
     TITLE = {A variational approach to the {Y}au-{T}ian-{D}onaldson
              conjecture},
   JOURNAL = {J. Amer. Math. Soc.},
  FJOURNAL = {Journal of the American Mathematical Society},
    VOLUME = {34},
      YEAR = {2021},
    NUMBER = {3},
     PAGES = {605--652},
      ISSN = {0894-0347},
   MRCLASS = {32Q20 (14E30 32P05 32Q26 32U35)},
  MRNUMBER = {4334189},
       DOI = {10.1090/jams/964},
       URL = {https://doi.org/10.1090/jams/964},
}

@article {RWN14,
    AUTHOR = {Ross, Julius and Witt Nystr\"{o}m, David},
     TITLE = {Analytic test configurations and geodesic rays},
   JOURNAL = {J. Symplectic Geom.},
  FJOURNAL = {The Journal of Symplectic Geometry},
    VOLUME = {12},
      YEAR = {2014},
    NUMBER = {1},
     PAGES = {125--169},
      ISSN = {1527-5256},
   MRCLASS = {32L05 (32U05)},
  MRNUMBER = {3194078},
MRREVIEWER = {Bianca Santoro},
       DOI = {10.4310/JSG.2014.v12.n1.a5},
       URL = {https://doi.org/10.4310/JSG.2014.v12.n1.a5},
       label={RWN}
}

@article {DDNL19log,
    AUTHOR = {Darvas, Tam\'{a}s and Di Nezza, Eleonora and Lu, Chinh H.},
     TITLE = {Log-concavity of volume and complex {M}onge-{A}mp\`ere equations
              with prescribed singularity},
   JOURNAL = {Math. Ann.},
  FJOURNAL = {Mathematische Annalen},
    VOLUME = {379},
      YEAR = {2021},
    NUMBER = {1-2},
     PAGES = {95--132},
      ISSN = {0025-5831},
   MRCLASS = {32W20 (32Q15 32U40 52A20 53C55)},
  MRNUMBER = {4211083},
MRREVIEWER = {Rafa\l  Czy\.{z}},
       DOI = {10.1007/s00208-019-01936-y},
       URL = {https://doi.org/10.1007/s00208-019-01936-y},
label={DDNL}
}

@article {Cao14,
    AUTHOR = {Cao, Junyan},
     TITLE = {Numerical dimension and a {K}awamata-{V}iehweg-{N}adel-type
              vanishing theorem on compact {K}\"{a}hler manifolds},
   JOURNAL = {Compos. Math.},
  FJOURNAL = {Compositio Mathematica},
    VOLUME = {150},
      YEAR = {2014},
    NUMBER = {11},
     PAGES = {1869--1902},
      ISSN = {0010-437X},
   MRCLASS = {32L20 (14C20)},
  MRNUMBER = {3279260},
MRREVIEWER = {Tsz On Mario Chan},
       DOI = {10.1112/S0010437X14007398},
       URL = {https://doi.org/10.1112/S0010437X14007398},
}

@article {DPS01,
    AUTHOR = {Demailly, Jean-Pierre and Peternell, Thomas and Schneider,
              Michael},
     TITLE = {Pseudo-effective line bundles on compact {K}\"{a}hler manifolds},
   JOURNAL = {Internat. J. Math.},
  FJOURNAL = {International Journal of Mathematics},
    VOLUME = {12},
      YEAR = {2001},
    NUMBER = {6},
     PAGES = {689--741},
      ISSN = {0129-167X},
   MRCLASS = {32J27 (32Q15 32Q57)},
  MRNUMBER = {1875649},
MRREVIEWER = {Christophe Mourougane},
       DOI = {10.1142/S0129167X01000861},
       URL = {https://doi.org/10.1142/S0129167X01000861},
}

@article {BFJ08,
    AUTHOR = {Boucksom, S\'{e}bastien and Favre, Charles and Jonsson, Mattias},
     TITLE = {Valuations and plurisubharmonic singularities},
   JOURNAL = {Publ. Res. Inst. Math. Sci.},
  FJOURNAL = {Kyoto University. Research Institute for Mathematical
              Sciences. Publications},
    VOLUME = {44},
      YEAR = {2008},
    NUMBER = {2},
     PAGES = {449--494},
      ISSN = {0034-5318},
   MRCLASS = {32U25 (13A18 14B05 32U05)},
  MRNUMBER = {2426355},
MRREVIEWER = {Romain Dujardin},
       DOI = {10.2977/prims/1210167334},
       URL = {https://doi.org/10.2977/prims/1210167334},
}

@article {BEGZ10,
    AUTHOR = {Boucksom, S\'{e}bastien and Eyssidieux, Philippe and Guedj,
              Vincent and Zeriahi, Ahmed},
     TITLE = {Monge-{A}mp\`ere equations in big cohomology classes},
   JOURNAL = {Acta Math.},
  FJOURNAL = {Acta Mathematica},
    VOLUME = {205},
      YEAR = {2010},
    NUMBER = {2},
     PAGES = {199--262},
      ISSN = {0001-5962},
   MRCLASS = {32U40 (32Q20 32U15 32W20)},
  MRNUMBER = {2746347},
MRREVIEWER = {S\l awomir Dinew},
       DOI = {10.1007/s11511-010-0054-7},
       URL = {https://doi.org/10.1007/s11511-010-0054-7},
}

@article {GZ07,
    AUTHOR = {Guedj, Vincent and Zeriahi, Ahmed},
     TITLE = {The weighted {M}onge-{A}mp\`ere energy of quasiplurisubharmonic
              functions},
   JOURNAL = {J. Funct. Anal.},
  FJOURNAL = {Journal of Functional Analysis},
    VOLUME = {250},
      YEAR = {2007},
    NUMBER = {2},
     PAGES = {442--482},
      ISSN = {0022-1236},
   MRCLASS = {32W20 (32Q15 32U05)},
  MRNUMBER = {2352488},
MRREVIEWER = {Norman Levenberg},
       DOI = {10.1016/j.jfa.2007.04.018},
       URL = {https://doi.org/10.1016/j.jfa.2007.04.018},
}

@article {DDNL18mono,
    AUTHOR = {Darvas, Tam\'{a}s and Di Nezza, Eleonora and Lu, Chinh H.},
     TITLE = {Monotonicity of nonpluripolar products and complex
              {M}onge-{A}mp\`ere equations with prescribed singularity},
   JOURNAL = {Anal. PDE},
  FJOURNAL = {Analysis \& PDE},
    VOLUME = {11},
      YEAR = {2018},
    NUMBER = {8},
     PAGES = {2049--2087},
      ISSN = {2157-5045},
   MRCLASS = {32W20 (32Q15 32Q20 32U05 35J96)},
  MRNUMBER = {3812864},
MRREVIEWER = {S\l awomir Ko\l odziej},
       DOI = {10.2140/apde.2018.11.2049},
       URL = {https://doi.org/10.2140/apde.2018.11.2049},
       label={DDNL}
}

@book {MM07,
    AUTHOR = {Ma, Xiaonan and Marinescu, George},
     TITLE = {Holomorphic {M}orse inequalities and {B}ergman kernels},
    SERIES = {Progress in Mathematics},
    VOLUME = {254},
 PUBLISHER = {Birkh\"{a}user Verlag, Basel},
      YEAR = {2007},
     PAGES = {xiv+422},
      ISBN = {978-3-7643-8096-0},
   MRCLASS = {32L20 (32A25 58J20 58J35 58J52 58J60)},
  MRNUMBER = {2339952},
MRREVIEWER = {David Borthwick},
}

@book {Dem12,
    AUTHOR = {Demailly, Jean-Pierre},
     TITLE = {Analytic methods in algebraic geometry},
    SERIES = {Surveys of Modern Mathematics},
    VOLUME = {1},
 PUBLISHER = {International Press, Somerville, MA; Higher Education Press,
              Beijing},
      YEAR = {2012},
     PAGES = {viii+231},
      ISBN = {978-1-57146-234-3},
   MRCLASS = {32-02 (14C30 14F18 32J25 32Q15 32U40)},
  MRNUMBER = {2978333},
MRREVIEWER = {Valentino Tosatti},
}

@book {Ful,
    AUTHOR = {Fulton, William},
     TITLE = {Intersection theory},
    SERIES = {Ergebnisse der Mathematik und ihrer Grenzgebiete. 3. Folge. A
              Series of Modern Surveys in Mathematics [Results in
              Mathematics and Related Areas. 3rd Series. A Series of Modern
              Surveys in Mathematics]},
    VOLUME = {2},
   EDITION = {Second},
 PUBLISHER = {Springer-Verlag, Berlin},
      YEAR = {1998},
     PAGES = {xiv+470},
      ISBN = {3-540-62046-X},
   MRCLASS = {14C17 (14-02)},
  MRNUMBER = {1644323},
       DOI = {10.1007/978-1-4612-1700-8},
       URL = {https://doi.org/10.1007/978-1-4612-1700-8},
}

@phdthesis{Bou02,
  title={C{\^o}nes positifs des vari{\'e}t{\'e}s complexes compactes},
  author={Boucksom, S.},
  year={2002},
  school={Universit{\'e} Joseph-Fourier-Grenoble I}
}

@inproceedings{Wlo09,
  title={Resolution of singularities of analytic spaces},
  author={W{\l}odarczyk, J.},
  booktitle={Proceedings of G{\"o}kova Geometry-Topology Conference 2008, G{\"o}kova Geometry/Topology Conference (GGT)},
  pages={31--63},
  year={2009}
}

Mingchen Xia, \textsc{Chalmers Tekniska Högskola and Institute of Geometry and Physics, USTC}\par\nopagebreak
  \textit{Email address}, \texttt{xiamingchen2008@gmail.com}\par\nopagebreak
  \textit{Homepage}, \url{https://mingchenxia.github.io/}.

\end{document}